\documentclass[oneside, 11pt]{amsart}
\pdfoutput=1
\usepackage{amsmath, mathtools, amssymb, amsthm, amscd, alltt, graphicx, comment}
\usepackage[utf8]{inputenc}
\usepackage[T1]{fontenc}
\usepackage[breaklinks=true,unicode]{hyperref}
\usepackage[capitalise]{cleveref}
\usepackage[matrix,arrow,curve,tips]{xy}
\usepackage{scalerel}

\SelectTips{cm}{12}

\usepackage[backend=biber, bibencoding=utf8, giveninits=true, citestyle=numeric-comp, sortlocale=en_US, url=false, doi=false, eprint=true, maxbibnames=4]{biblatex}
\renewbibmacro*{volume+number+eid}{\ifentrytype{article}{\- \iffieldundef{volume}{}{Vol.~\printfield{volume},}\iffieldundef{number}{}{ No.~\printfield{number},}}}
\renewbibmacro{in:}{\ifentrytype{article}{}{\printtext{\bibstring{in}\intitlepunct}}}
\newbibmacro{string+doi}[1]{\iffieldundef{doi}{\iffieldundef{url}{#1}{\href{\thefield{url}}{#1}}}{\href{https://dx.doi.org/\thefield{doi}}{#1}}}
\DeclareFieldFormat[article, inproceedings, inbook, book, online]{title}{\usebibmacro{string+doi}{\mkbibquote{#1}}}

\numberwithin{equation}{section}
\newtheorem{lemma}{Lemma} \numberwithin{lemma}{section}
\newtheorem{prop}[lemma]{Proposition}
\newtheorem{theorem}{Theorem}
\newtheorem{corollary}[lemma]{Corollary} 
\newtheorem*{theorem*}{Theorem} 
\newtheorem*{corollary*}{Corollary}

\theoremstyle{definition} 
\newtheorem{df}[lemma]{Definition} \Crefname{df}{Definition}{Definitions}
\newtheorem{example}[lemma]{Example} \Crefname{example}{Example}{Examples}

\theoremstyle{remark} 
\newtheorem{rem}[lemma]{Remark}
\newtheorem{conv}[lemma]{Convention} \Crefname{conv}{Convention}{Conventions}

\DeclareMathOperator\St{St}
\DeclareMathOperator\Ker{Ker}
\DeclareMathOperator\GG{G}
\DeclareMathOperator\E{E}

\DeclareMathOperator{\Pro}{Pro}
\DeclareMathOperator\Aut{Aut}

\newcommand{\Set}{\mathbf{Set}}
\newcommand{\Group}{\mathbf{Grp}}
\newcommand{\Rng}{\mathbf{Rng}}
\newcommand{\Fun}{\mathbf{Fun}}
\newcommand{\Mod}{\mathbf{Mod}}
\newcommand{\op}{\mathrm{op}}
\newcommand{\ZZ}{\mathbb{Z}}

\newcommand{\otimeshat}{\mathbin{\widehat{\otimes}}}

\newcommand{\up}[2]{{^{#1}\!{#2}}}

\newcommand{\rA}{\mathsf{A}}
\newcommand{\rB}{\mathsf{B}}
\newcommand{\rC}{\mathsf{C}}
\newcommand{\rD}{\mathsf{D}}
\newcommand{\rE}{\mathsf{E}}
\newcommand{\rF}{\mathsf{F}}

\title{Centrality of $\mathrm K_2$ for Chevalley groups: a pro-group approach}
\keywords {Steinberg group, pro-group, Suslin normality theorem. {\em Mathematical Subject Classification (2010):} 19C09, 20G35, 20H05}

\author[1] {Andrei Lavrenov} \email{avlavrenov at gmail.com} 
\author[2] {Sergey Sinchuk} \email{sinchukss at gmail.com}
\author[3] {Egor Voronetsky} \email{voronetckiiegor at yandex.ru}
\address{St. Petersburg State University, 14th Line V.O., 29B, Saint Petersburg 199178 Russia}

\date {\today}

\begin{document}
\maketitle
\begin{abstract} 
We prove the centrality of $\mathrm K_2 (\rF_4, \,R)$ for an arbitrary commutative ring $R$. This completes the proof of the centrality of $\mathrm K_2(\Phi,\, R)$ for any root system $\Phi$ of rank $\geq 3$. Our proof uses only elementary localization techniques reformulated in terms of pro-groups. 
Another new result of the paper is the construction of a crossed module on the canonical homomorphism $\St(\Phi, R) \to \GG_\mathrm{sc}(\Phi, R)$, which has not been known previouly for exceptional $\Phi$. \end{abstract}

\section{Introduction}
The aim of this paper is to give a uniform proof of the centrality of the $\mathrm{K}_2$-functor modeled on Chevalley groups for an arbitrary commutative ring.
Our proof is based on the technique of pro-groups introduced by the third-named author in~\cite{Vor1} and~\cite{Vor2}.
Recall that~\cite{Vor1} is dedicated to the centrality of the linear $\mathrm{K}_2$ for not necessarily commutative rings while~\cite{Vor2} focuses on the study of the centrality problem in the context of {\it odd-unitary groups}, a large class of groups generalizing the usual classical groups; see~\cite{Pe05}.

Let us briefly explain the reader how our results fit into the general picture.
Recall that Steinberg groups $\St(\Phi, R)$ are certain groups given by generators and relations, that were classically introduced by R.~Steinberg and J.~Milnor in~\cite{St62, Milnor} as combinatorial approximations of the simply-connected Chevalley groups. For example, every element of the special linear group $\mathrm{SL}_n(\mathbb F_q)$ over a finite field can be presented as a product of elementary transvections $t_{ij}(a)$ ($i\neq j$, $a\in \mathbb F_q$). These elements are taken as the generators of the Steinberg group $\St_n(\mathbb{F}_q)$. 
Moreover, any two such presentations of the same element of $\mathrm{SL}_n(\mathbb{F}_q)$ can be rewritten one into the other
 via a sequence of the following ``elementary'' relations: 
\begin{itemize} 
 \item $t_{ij}(a)t_{ij}(b) = t_{ij}(a+b)$; 
 \item $t_{ij}(a)t_{hk}(b) = t_{hk}(b)t_{ij}(a)$ for $j\neq h$ and $i\neq k$; 
 \item $t_{ij}(a)t_{jk}(b)t_{ij}(-a)t_{jk}(-b) = t_{ik}(ab)$ for $i\neq k$.
\end{itemize} 
These relations are called {\it Steinberg relations} and are taken as the defining relations of the Steinberg group $\St_n(\mathbb{F}_q)$. Thus, in the finite field case one has $\St_n(\mathbb{F}_q) \cong \mathrm{SL}_n(\mathbb{F}_q)$.

For an infinite field $F$ the Steinberg group $\mathrm{St}(\Phi,\,F)$ no longer coincides with the simply-connected Chevalley group $\GG_{\mathrm{sc}}(\Phi, F)$, but rather is the universal central extension of the latter.
By the classical theory of central extensions the kernel of the natural homomorphism $\St(\Phi,\,F)\to\GG_{\mathrm{sc}}(\Phi, F)$ coincides with the Schur multiplier $\mathrm H_2(\mathrm G_{\mathrm{sc}}(\Phi,\,F),\,\mathbb Z)$. The concrete presentation of this kernel in the field case was obtained by H.~Matsumoto in~\cite{Ma69}. Similar presentation for the special linear group over a skew field also has been obtained by U.~Rehmann in~\cite{Re78}.

Recall that the algebraic $\mathrm K_2$-functor was originally defined by J.~Milnor as the kernel of the homomorphism $\St_\infty(R) \to \mathrm{GL}_\infty(R)$, i.\,e. as the group of ``nontrivial'' relations between transvections in the infinite-dimensional general linear group. Also he showed that the Schur multipliers of the linear Steinberg group $\St_n(R)$ are trivial for $n \geq 5$ and an arbitrary $R$. This shows that $\St_n(R)$ is a good candidate for the role of the universal central extension of the elementary subgroup $\mathrm{E}_n(R) = \mathrm{Im}(\St_n(R) \to \mathrm{SL}_n(R))$.

Milnor's approach was followed by M.~Stein who defined Steinberg groups $\St(\Phi, R)$ in the context of arbitrary Chevalley groups and, by analogy,  defined the functor $\mathrm K_2(\Phi, R)$ as the kernel of the natural homomorphism \begin{equation}\label{eq:st}\mathrm{st}\colon \St(\Phi, R) \to \GG_{\mathrm{sc}}(\Phi, R).\end{equation} Also M.~Stein and W.~van der Kallen computed Schur multipliers of $\St(\Phi, R)$ for $\Phi$ of rank $\geq 3$; see~\cite{St71, vdKStein}. The Schur multipliers in the setting of unitary Steinberg groups over noncommutative rings have also been studied in~\cite{Bak, Tang, LavU, LavOU}.

The first proof of the centrality of the linear functor \[\mathrm K_{2,n}(R) = \Ker(\St_n(R) \to \mathrm{SL}_n(R))\] for an {\it arbitrary} commutative ring $R$ was obtained by W.~van der Kallen in~\cite{vdK} for $n \geq 4$. His proof was influenced by the proof of the Suslin normality theorem, which asserts that the elementary subgroup $\mathrm E_n(R)$ is normal in $\mathrm{SL}_n(R)$ for an arbitarary commutative $R$ and $n \geq 3$. The latter result is one of the ingredients in Suslin's proof of the so-called $\mathrm K_1$-analogue of the Serre problem; see~\cite{Suslin}. The result of van der Kallen was further generalized by M.~Tulenbaev to almost commutative rings, i.\,e. algebras module finite over their centers; see~\cite{Tul}.

While the centrality of $\mathrm K_2(\Phi, R)$ has been long known for rings of small Krull dimension, e.\,g. for semilocal rings (see~\cite{Ste73}), the question whether it holds for an {\it arbitrary} commutative ring has remained open since~\cite{vdK}.
On the other hand, the analogue of Suslin's normality theorem for Chevalley groups of rank $\geq 2$, i.\,e. the normality of the subgroup $\E_{\mathrm{sc}}(\Phi, R):= \mathrm{Im}(\mathrm{st})$ for an arbitrary commutative ring $R$, was soon obtained by G.~Taddei, see~\cite{Ta86}. Since then the normality theorem was generalized to even larger classes of groups e.\,g. isotropic reductive groups and odd unitary groups; see~\cite{Pe05,PS09,Ste16}.

The first advancement in the solution of the centrality problem for arbitrary commutative rings since~\cite{vdK} was the counterexample of M.~Wendt~\cite{Wendt} which showed that the centrality of $\mathrm K_2$ may fail for root systems of rank $2$ (similarly, there is a counterexample to Suslin's normality theorem for the rank $1$ group $\mathrm{SL}_2(R)$). Soon the papers~\cite{Lav,Sin,LavSin} of the first- and second-named authors appeared, in which it was shown that the centrality of $\mathrm K_2$ does, indeed, hold for all the Chevalley groups of type $\rC_\ell$, $\rD_\ell$, $\rE_\ell$ provided $\ell\geq 3$. In~\cite{Lav} a symplectic analogue of the technique of~\cite{vdK} was developed, while in~\cite{Sin, LavSin} the proof was based on the amalgamation theorem for relative Steinberg groups which reduced the problem of centrality to the already-known linear case. Nevertheless, neither of these two approaches seemed to work for the root systems $\rB_\ell$ and $\rF_4$. Finally, in~\cite{Vor1} the third-named author has developed a novel pro-group approach, which allowed him not only to generalize~\cite{vdK,Tul} to even larger class of noncommutative rings, but also to prove the centrality of $\mathrm K_2$ for odd unitary groups and, thus, cover all classical groups, including the case $\rB_\ell$, see~\cite{Vor2}. In the present paper these methods are employed to verify the centrality of $\mathrm K_2$ in the case $\Phi = \rF_4$, the only case which has remained open until now.
This is accomplished in~\cref{SteinbergCrossedModule} below. Our method of proof also applies to root systems of type $\rA_{\geq 3}$, $\rD_{\geq 4}$, $\rE_{6,7,8}$, and thus gives a proof of the centrality of $\mathrm{K}_2$ independent of the earlier papers~\cite{vdK,Sin,LavSin}. In the course of the proof we also reprove Taddei's normality theorem~\cite{Ta86} in the aforementioned cases; see~\cref{Normality}. Thus, the main consequence of \cite{Vor2} and our~\cref{SteinbergCrossedModule} can be formulated as follows.
\begin{theorem*}
Let $R$ be a commutative ring, $\Phi$ be an irreducible root system of rank $\geq 3$. 
Then the map $\mathrm{st}\colon\mathrm{St}(\Phi,\,R) \to \mathrm{E}_{\mathrm{sc}}(\Phi,\,R)$ is a central extension.
\end{theorem*}

In fact, in~\cref{SteinbergCrossedModule} we prove a stronger result. More precisely, we endow the canonical homomorphism~\eqref{eq:st} with the structure of a crossed module; see~\cref{df:crossed-module}. This is yet another new result of the paper. Notice that no construction of a crossed module on~\eqref{eq:st} was presented in~\cite{Lav, Sin, LavSin}. Some of the material of the present paper (e.\,g.~\cref{sec:local-Chevalley}) is required only for this construction and is not needed for the proof of the centrality of $\mathrm K_2$.

Using the above theorem we obtain a result comparing Stein's $\mathrm{K}_2$-groups modeled on Chevalley groups with {\it Quillen's} unstable $\mathrm{K}_2$-groups.
Recall that the latter are defined by means of Quillen's $+$-construction: for a root system $\Phi$ of
rank $\geq3$ take the perfect normal subgroup $\mathrm{E}_{\mathrm{sc}}(\Phi, R) \trianglelefteq \mathrm{G}_{\mathrm{sc}}(\Phi, R)$ and set \[ \mathrm{K}_2^Q(\Phi, R) := \pi_2\mathrm{BG}_{\mathrm{sc}}(\Phi,\,R)^+_{\mathrm{E}_{\mathrm{sc}}(\Phi, R)}.\] The above group also coincides with the homology group $\mathrm H_2(\mathrm E_{\mathrm{sc}}(\Phi,\,R),\,\mathbb Z)$; see e.\,g.~\cite[\S~IV.1]{Weibel}.
Thus, combining the above theorem with~\cite{St71,vdKStein} we obtain the following result.
\begin{corollary*}
For an arbitrary commutative ring $R$ and a root system $\Phi$ of rank $\geq 3$ there is an exact sequence \[ \xymatrix{ 1\ar@{->}[r]&{\mathrm H_2(\St(\Phi,\,R),\ZZ)}\ar@{->}[r]&\mathrm K_2^Q(\Phi,\,R)\ar@{->}[r]&\mathrm K_2(\Phi,\,R)\ar@{->}[r]&1,}\] in which the group $\mathrm H_2(\mathrm{St}(\Phi, R), \ZZ)$ is trivial in the cases 
\[\Phi= \rA_{\geq 4},\rB_{\geq 4},\rC_{\geq 4}, \rD_{\geq 5}, \rE_{6,7,8},\] or is one of the groups in the following list:
\begin{align*}
\mathrm H_2\big(\St(\rA_3,\,R),\ZZ\big)&=R/\langle2, (t^2-t)(s^2-s)\mid t,\,s\in R\rangle,\\
\mathrm H_2\big(\St(\rB_3,\,R),\ZZ\big)&=R/\langle6, 3(t^2-t)(s^2-s), 2(t^3-t)\mid t,\,s\in R\rangle,\\
\mathrm H_2\big(\St(\rC_3,\,R),\ZZ\big)&=R/\langle t^2-t\mid t\in R\rangle,\\
\mathrm H_2\big(\St(\rD_4,\,R),\ZZ\big)&=\big(R/\langle t^2-t\mid t\in R\rangle\big)\times\big(R/\langle2,(t^2-t)(s^2-s)\mid t,s\in R\rangle\big),\\
\mathrm H_2\big(\St(\rF_4,\,R),\ZZ\big)&=R/\langle t^2-t\mid t\in R\rangle.
\end{align*}
In particular, for $\Phi=\rA_3,\rC_3,\rD_4,\rF_4$ the equality $\mathrm K_2(\Phi,\,R)=\mathrm K_2^Q(\Phi,\,R)$ holds iff $R$ does not have residue fields isomorphic to $\mathbb F_2$ and $\mathrm K_2(\mathrm B_3,\,R)=\mathrm K_2^Q(\mathrm B_3,\,R)$ iff $R$ does not have residue fields isomorphic to $\mathbb F_2$ or $\mathbb F_3$.
\end{corollary*}

Since the proof of the main theorem is rather long and technical and the setting of the papers~\cite{Vor1,Vor2} is already rather general, let us briefly sketch the key ideas of our proof assuming, for simplicity, that the root system $\Phi$ is simply-laced and the ring $R$ is an integral domain.

The proof is based on a variant of Quillen--Suslin's local-global principle.
Consider an arbitrary element $g$ from $\mathrm K_{2}(\Phi, R)$. We want to show that $g$ commutes with $x_{\alpha}(r)$ for every $r\in R$ and $\alpha\in\Phi$.
For $\alpha \in \Phi$ consider the set \[I_\alpha=\{a\in R\mid [g,\,x_{\alpha}(aR)]=1\},\] which is easily seen to be, in fact, an ideal.
To prove the centrality of $\mathrm{K}_{2}(\Phi, R)$ it suffices to show that $I_\alpha=R$ for all $\alpha\in \Phi$.
Assume the contrary and consider an arbitrary maximal ideal $M \trianglelefteq R$ containing $I_\alpha$ for some $\alpha$. 
By the centrality of $\mathrm K_{2}(\Phi, R_M)$ one has the identity $[\lambda_{M}(g),\,x_{\alpha}(r)]=1$ for all $r\in R_{M}$ (we denote by $\lambda_M$ the localization homomorphism $R \to R_M$). 
If we were able to show for some $s_0\in R\setminus M$ that the identity \begin{equation} \label{intro} [g,\,x_{\alpha}(s_0R)]=1, \end{equation} holds in $\St(\Phi, R)$ or, in other words, that $s_0\in I_\alpha\cap R\setminus M$, this would lead to a contradiction and hence would complete the proof of the centrality of $\mathrm K_{2}(\Phi, R)$. Unfortunately, formally proving this is not easy. 

In the earlier papers the proof went as follows; cf.~\cite{Tul,Sin,LavSin,Lavlgp}.
First, a formal variable was thrown in and the original question was reduced to the question of the triviality of $[g, x_{\alpha}(t)]$ in $\St(\Phi, R[t])$, which, in turn, was further reduced to proving the {\it Dilation principle}, an assertion about principal localizations similar to~\eqref{intro} (see e.\,g.~\cite[Lemma~15]{Sin}). The proof of the Dilation principle was based on the isomorphism of {\it relative Steinberg groups} $\St(\Phi, R_a[t], tR_a[t]) \cong \St(\Phi, R \ltimes tR_a[t], tR_a[t])$; cf.~\cite[Remark~3.12]{Sin}. Essentially this isomorphism means that it is possible to define a conjugation action of $\St(\Phi, R_a)$ on the group $\St(\Phi, R\ltimes tR_a[t], tR_a[t])$.
This last step, however, required finding nontrivial ad hoc presentations of relative Steinberg groups.
For example~\cite{Lavlgp} relied on a variant of van der Kallen's ``another presentation''~\cite{vdK} for symplectic groups, while~\cite{Sin, LavSin} were based on the amalgamation theorem for relative Steinberg groups; see~\cite[Theorem~9]{Sin}. However, it was not easy to transfer to all Chevalley groups either of these two approaches.

A remarkable idea of the third-named author was to abandon relative Steinberg groups completely and instead work with sequences of {\it unrelativized} Steinberg groups $\St(\Phi, a^nR)$, $n \in \mathbb N$ given by the usual Steinberg presentations.
It turns out that it is possible to naturally define the ``conjugation'' action of the Steinberg group $\St(\Phi, R_a)$ on such a sequence. Naively, for $g \in \St(\Phi, R_a)$ the automorphism of conjugation by $g$ corresponds to a collection of group homomorphisms \[\{\St(\Phi, a^{c^*(m)}R) \to \St(\Phi, a^{m}R)\}_{m \in \mathbb N },\] where $c^*\colon \mathbb N \to \mathbb N$ is some function depending on $g$.

In fact, it is possible to even further generalize this construction to arbitrary localizations, not just principal ones.
Indeed, if $S \subseteq R$ is an arbitrary multiplicative subset of a domain $R$, then one can consider the following diagram.
Its objects are unrelativized Steinberg groups $\St(\Phi, sR)$, $s \in S$, while its arrows are the homomorphisms
\[ \St(\Phi, s_2R) \to \St(\Phi, s_1R),\ x_\alpha(a) \mapsto x_\alpha(a), \]
which may exist, obviously, only if $s_1$ divides $s_2$. 
The above diagram is called the {\it Steinberg pro-group} and its arrows are called the {\it structure morphisms}. 
The key claim is that it is possible to define an action of $\St(\Phi, S^{-1}R)$ on the Steinberg pro-group. 

For each $x_{\alpha}(a/s)$, $a\in R$, $s\in S = R\setminus M$ the automorphism of conjugation by $x_\alpha(a/s)$ corresponds to a collection of group homomorphisms \begin{equation} \label{eq:collection} \{\mathrm{conj}_{x_\alpha(a/s)}(t) \colon  \St(\Phi, c^*(t)R) \to \St(\Phi, tR)\}_{t \in S},\end{equation} where $c^*\colon S \to S$ is the function $t \mapsto ts$, called the {\it index function}. The construction of the homomorphisms $\mathrm{conj}_{x_\alpha(a/s)}(t)$ will be outlined below.

In fact, a morphism between pro-groups should be defined not just as a collection of group homomorphisms, but as an {\it equivalence class} of such collections. By definition, two pro-group morphisms $m_1$ and $m_2$ with index functions $m_1^*$,$m_2^* \colon S \to S$ are equivalent if for every $t$ one can choose $s$ divisible by both $m_1^*(t)$ and $m_2^*(t)$ so that the homomorphisms \[m_i(t) \colon \St(\Phi, m_i^*(t)R) \to \St(\Phi, tR),\ i=1,2\] coincide after precomposition with the structure morphisms \[\St(\Phi, sR) \to \St(\Phi, m^*_i(t)R),\ i=1,2.\]
We refer the reader to~\cref{sec:pro-obj} for a more thorough treatment of the formalism of pro-groups.

After the automorphisms of conjugation by $x_\alpha(a/s)$ are defined, we verify that they satisfy the Steinberg relations defining $\St(\Phi, S^{-1}R)$ and, moreover, that the Steinberg symbols $\{b_1, b_2\}$, $b_1, b_2 \in {S^{-1}R}^\times$ act trivially on the pro-group. Technically, this is an easier step since it is enough to verify the corresponding equalities on suitable group generators.
Now, for $g = \prod x_{\alpha_i}(a_i)$ the pro-group automorphism $\mathrm{conj}_{\lambda_M(g)}$ can be defined as the composition of $\mathrm{conj}_{x_{\alpha_i}(a_i/1)}$.
A key point here is that the resulting pro-group automorphism is compatible with the usual conjugation action of $g \in \St(\Phi, R)$, i.\,e. for some $s_0 \in S$ the group homomorphism \[\mathrm{conj}_{\lambda_M(g)}(1) \colon \St(\Phi, s_0R) \to \St(\Phi, R)\] coincides with the obvious homomorphism $h \mapsto ghg^{-1}$. Since our conjugation action factorizes through $\St(\Phi, S^{-1}R)$ we conclude that the conjugation by $g \in \mathrm K_2(\Phi, R)$ acts trivially on $x_\alpha(s_0R)$, which finishes the proof of~\eqref{intro}. This part of the proof is formalized in~\cref{sec:local-action,sec:proof-main}.

Now let us focus on the construction of the action of $\St(\Phi, S^{-1}R)$ on the Steinberg pro-group. In order to construct the automorphism of conjugation by $x_{\alpha}(a/s)$ we need to introduce the group $\St(\mathrm \Phi\setminus\{-\alpha\}, tR)$ whose presentation is obtained from the standard presentation of the Steinberg group $\St(\Phi, tR)$ by omitting all the generators $x_{-\alpha}(tr)$, $r \in R$ and all the relations involving such generators. Now the analogues of the homomorphisms~\eqref{eq:collection} between such groups can be constructed fairly easily based on the remaining part of the Steinberg presentation (below $N$ denotes a sufficiently large natural number):
\begin{equation} \label{eq:conj-def} x_\beta(s^Nb) \mapsto \left\{\begin{array}{ll} x_{\alpha+\beta}(N_{\alpha,\beta} s^{N-1}ab) x_\beta(s^Nb) & \text{if $\alpha+\beta\in\Phi$,}\\ x_\beta(s^Nb) & \text{if $\alpha+\beta\not\in\Phi\cup \{0\}$.}\end{array}\right. \end{equation}
Notice that there is no simple expression similar to the right-hand side of~\eqref{eq:conj-def} in the case $\beta=-\alpha$, which explains our need to use groups $\St(\Phi\setminus\{-\alpha\}, tR)$ as domains for the homomorphisms~\eqref{eq:collection}.
If follows from~\cref{endomor-elim}) that~\eqref{eq:conj-def} defines a collection of well-defined homomorphisms:
\[\mathrm{conj}'_{x_\alpha(r/s)}(t)\colon \St(\Phi\setminus\{-\alpha\}, s^N tR) \to \St(\Phi\setminus\{-\alpha\}, tR).\]

The obvious defect of the formula~\eqref{eq:conj-def} is that its correctness depends on the integrality of $R$.
However, it is easy to fix this using the notion of a {\it homotope of a ring} introduced in~\cite{Vor1,Vor2}. By definition, for $t \in R$ the {\it $t$-homotope} $R^{(t)}$ of $R$ is the ring without unit, multiplication in which is given by the formula $x*y=txy$. The homotope $R^{(t)}$ is isomorphic to $tR$ as a ring without unit if $R$ happens to be a domain. 
Now, replacing everywhere above the group $\St(\Phi, tR)$ with $\St(\Phi, R^{(t)})$, one can formulate the analogue of~\eqref{eq:conj-def} for not necessarily integral rings.
We believe that the proof which works for arbitrary commutative rings has extra elegance as compared to the proof which first reduces to the case of domains. However, the reader not comfortable with the notion of a ring homotope may follow the latter path and think of $R^{(t)}$ as just a synonym for $tR$; see~\cref{no-homotopes} for an outline of such a proof.

Although, the groups $\St(\Phi, tR)$ and $\St(\Phi\setminus\{-\alpha\}, tR)$ need not be isomorphic in general, the pro-groups composed of them turn out to be isomorphic in the category of pro-groups, see~\cref{SingleRootElimination}. Notice also that the corresponding assertion for unital rings, i.\,e. the isomorphism $\St(\Phi, R) \cong \St(\Phi\setminus\{-\alpha\}, R)$ is an easy corollary of a much stronger result called {\it Curtis--Tits} presentation, which allows one to omit most of the Steinberg generators from the presentation of $\St(\Phi, R)$; see e.\,g.~\cite[Corollary~1.3]{A13}. In fact, the main difference of the present paper from~\cite{Vor1, Vor2} is that our proof is based on Steinberg presentations omitting roots, while~\cite{Vor1, Vor2} rely on computations with relative root systems. 

If written in down-to-earth terms, the aforementioned isomorphism of pro-groups amounts to the following collection of group homomorphisms in the ``non-obvious'' direction ($M$ is a sufficiently large natural number):
\[\{F_\alpha(t)\colon \St(\Phi, t^{M}R)\rightarrow \St(\Phi\setminus\{-\alpha\}, tR) \}_{t\in S}.\]
Of course, the natural choice for the image of $x_\alpha(t^Mr)$ would be the commutator $[x_{\alpha-\beta}(t^{M-1}b),\,x_{\beta}(\pm t)]$ for some $\beta \in \Phi$ such that $\alpha-\beta\in \Phi$.
The fact that this definition does not depend on the choice of $\beta$ can be checked using the Hall--Witt identity~\eqref{eq:HW} (notice that we can further decompose $x_{\alpha-\beta}(t^{M-1}b)$ into a commutator in $\St(\Phi, tR)$).
Similarly, we can verify that these commutators satisfy the relations missing from the presentation of $\St(\Phi\setminus\{-\alpha\}, t^MR)$. Keeping track of index functions and individual homomorphisms $F_\alpha(t)$ in such computations, however, is not very convenient, so, to simplify the proof of~\cref{SingleRootElimination}, we prefer to write down all commutator formula calculations inside the category of pro-groups. In addition, to reduce the number of cases that need to be considered, we also assume throughout the paper that $\Phi$ is not of type $\rB_\ell$ or $\rC_\ell$. The latter assumption makes it possible to decompose any root of $\Phi$ into a sum of roots having the same length. Now, composing $F_\alpha(t)$ and $\mathrm{conj}'_{x_\alpha(r/s)}(t)$ in a suitable way we obtain the sought pro-group morphism $\mathrm{conj}_{x_\alpha(r/s)}$.

\subsection{Acknowledgements}
The authors of the paper would like to express their sincere gratitude to the anonymous referee for a careful reading of the first version of this paper and numerous helpful suggestions.

The authors of the paper were supported by ``Native towns'', a social investment program of PJSC ``Gazprom Neft''. The first-named author was also supported by the Ministry of Science and Higher Education of the Russian Federation, agreement No. 075-15-2019-1619. The second-named author was also supported by RFBR grant No. 18-31-20044. 
The second- and third-named authors were also supported by the Foundation for the Advancement of Theoretical Physics and Mathematics ``BASIS``. The first-named author is a winner of Young Russian Mathematics contest and would like to thank its sponsors and the jury.

\section{Preliminaries}
Throughout this paper all commutators are left normed, i.\,e. $[x, y] = xyx^{-1}y^{-1}$. We denote by $x^y$ and $\up{y}x$ the elements $xyx^{-1}$ and $yxy^{-1}$.
In this paper we also make use of the following commutator identities:
\begin{align}
\label{eq:comm-mult-rhs}[x,\ yz]& =   [x,\ y] \cdot \up{y}[x,\ z],
\\ \label{eq:comm-mult-lhs}[xy,\ z]& = \up{x}[y,\ z] \cdot [x,\ z],
\end{align}
Recall that the Hall--Witt identity asserts that
\begin{equation} \label{eq:HW} [[x,\,y],\,\up{y} z] \cdot [[y,\,z],\,\up{z} x] \cdot [[z,\,x],\,\up{x} y] = 1. \end{equation}
As a corollary we obtain
\begin{equation} \label{eq:HW-corr} [x,z] = 1 \text{ implies } [x, [y,z]] = [[x,y],\up{y}z].
\end{equation}
Finally, it is not hard to check that
\begin{equation} \label{eq:comm-new}
[x, y] = 1 \text{ implies } [x, [y, z]] = [y, \up xz]\, [z, y].
\end{equation}

\subsection{Generalities on pro-objects} \label{sec:pro-obj}
Let \(\mathcal C\) be an arbitrary category.
In this section we recall the construction of the pro-completion of \(\mathcal C\) (cf. \cite[Section~6.1]{SK06}).

Recall that a nonempty small category \(\mathcal I\) is called {\it filtered} if
\begin{itemize}
 \item for any two objects \(i, k \in Ob(\mathcal{I})\) there is a diagram \(i \to j \leftarrow k\) in \(\mathcal I\) (i.\,e. $i$ and $k$ have an upper bound);
 \item every two parallel morphisms \(i \rightrightarrows j\) are equalized by some morphism \(j \to k\) in \(\mathcal I\).
\end{itemize}
A {\it pro-object} in \(\mathcal C\) is, by definition, a functor $X^{(\infty)}\colon \mathcal{I}_X^{\op} \to \mathcal{C}$, i.\,e. a contravariant functor from a filtered category \(\mathcal I_X\), called {\it the category of indices of $X^{(\infty)}$}, to the category \(\mathcal C\). 

We denote by \(X^{(i)}\) the value of \(X^{(\infty)}\) on an index \(i\).
The values of the functor $X^{(\infty)}$ on the arrows of $\mathcal{I}$ are called the {\it structure morphisms} of $X^{(\infty)}$.

The category of pro-objects is denoted by \(\Pro(\mathcal C)\). The hom-sets in this category are given by the formula
\begin{equation} \label{eq:pro-c-hom} \Pro(\mathcal C)(X^{(\infty)}, Y^{(\infty)}) = \varprojlim_{j \in \mathcal I_Y} \varinjlim_{i \in \mathcal I_X} \mathcal C(X^{(i)}, Y^{(j)}). \end{equation}
Let us recall a more explicit description of morphisms in \(\Pro(\mathcal C)\). 
By definition, a {\it pre-morphism} \(f \colon X^{(\infty)} \to Y^{(\infty)}\) consists of the following data:
\begin{itemize}
\item a set-theoretic function \(f^* \colon Ob(\mathcal{I_Y}) \to Ob(\mathcal{I_X})\);
\item a collection of morphisms \(f^{(i)} \colon X^{(f^*(i))} \to Y^{(i)}\) in $\mathcal{C}$ parametrized by $i \in Ob(\mathcal{I_Y})$. These morphisms are required to satisfy the following additional assumption: for every morphism \(i \to j\) in \(\mathcal I_Y\) there exists a sufficiently large index \(k \in Ob(\mathcal{I_X})\) such that the composite morphisms \(X^{(k)} \to X^{(f^*(i))} \to Y^{(i)}\) and \(X^{(k)} \to X^{(f^*(j))} \to Y^{(j)} \to Y^{(i)}\) are equal. \end{itemize}
The composition of two pre-morphisms \(f \colon X^{(\infty)} \to Y^{(\infty)}\) and \(g \colon Y^{(\infty)} \to Z^{(\infty)}\) is defined as the pre-morphism \(g \circ f\), where \((g \circ f)^*(i) = f^*(g^*(i))\) and \((g \circ f)^{(i)} = g^{(i)} \circ f^{(g^*(i))}\).

Two parallel pre-morphisms \(f, g \colon X^{(\infty)} \to Y^{(\infty)}\) are called equivalent if for every \(i \in Ob(\mathcal I_Y)\) there exists a sufficiently large index \(j \in Ob(\mathcal I_X)\) such that the composite morphisms \(X^{(j)} \to X^{(f^*(i))} \to Y^{(i)}\) and \(X^{(j)} \to X^{(g^*(i))} \to Y^{(i)}\) are equal. Finally, a {\it morphism} \(X^{(\infty)} \to Y^{(\infty)}\) is an equivalence class of pre-morphisms. Note that the equivalence relation is preserved by the composition operation.

There is a fully faithful functor $\mathcal{C} \to \Pro(\mathcal{C})$ sending \(X \in Ob(\mathcal C)\) to the pro-object $X \colon \mathbf{1}^\op \to \mathcal{C}$. It is clear from~\eqref{eq:pro-c-hom}  that
\begin{align}
 \Pro(\mathcal C)(X, Y) &\cong \mathcal C(X, Y);\\
 \Pro(\mathcal C)(X^{(\infty)}, Y) &\cong \varinjlim_{i \in \mathcal I_X} \mathcal C(X^{(i)}, Y); \label{eq:pro-c-hom-2}\\
 \Pro(\mathcal C)(X, Y^{(\infty)}) &\cong \varprojlim_{i \in \mathcal I_Y} \mathcal C(X, Y^{(i)}).
\end{align}

Moreover, it is also clear from~\eqref{eq:pro-c-hom} and~\eqref{eq:pro-c-hom-2} that the following assertion holds.
\begin{lemma} \label{lem:proobj-is-a-limit}
  \(X^{(\infty)}\) is the projective limit of \(X^{(i)}\) in the category \(\Pro(\mathcal C).\)
\end{lemma}

The category of pro-sets \(\Pro(\Set)\) has all finite limits by \cite[Prop.~6.1.18]{SK06} and therefore is a cartesian monoidal category.

Let us describe an explicit construction of limits in one important special case. Let $X^{(\infty)}$, $Y^{(\infty)}\colon\mathcal{I}^{\op}\to\Set$ be a pair of pro-sets with the same index category. Recall that the pointwise product of functors $X^{(\infty)}\times Y^{(\infty)}\colon\mathcal{I}^{\op}\to\Set$ is given by $(X^{(\infty)} \times Y^{(\infty)})^{(i)} = X^{(i)} \times Y^{(i)}$. Clearly, $X^{(\infty)} \times Y^{(\infty)}$ is a product of $X^{(\infty)}$ and $Y^{(\infty)}$ in the category $\Fun(\mathcal{I}^{\op}, {\Set})$. 

There is an obvious identity-on-objects functor $\Fun(\mathcal{I}^\op, \Set) \to \Pro(\Set)$ sending a natural transformation $\varphi \colon X^{(\infty)} \to Y^{(\infty)}$ to the morphism of pro-sets given by the pre-morphism $\varphi^* = \mathrm{id}_{Ob(\mathcal{I})}$, $\varphi^{(i)} = \varphi_i$. The diagonal morphism $\Delta \colon X^{(\infty)} \to X^{(\infty)} \times X^{(\infty)}$ and the canonical projection morphisms \(\pi_X\colon X^{(\infty)} \times Y^{(\infty)} \to X^{(\infty)}, \pi_Y\colon X^{(\infty)} \times Y^{(\infty)} \to Y^{(\infty)}\) can be defined in the category $\Pro(\Set)$ as the images of the corresponding natural transformations in $\Fun(\mathcal{I}^\op, \Set)$ under this functor.

We claim that the point-wise product $X^{(\infty)} \times Y^{(\infty)}$ satisfies the universal property of products in the category $\Pro(\Set)$. Indeed, this follows from~\eqref{eq:pro-c-hom} and the fact that filtered colimits commute with finite limits. This argument also shows that arbitrary finite limits of pro-sets with the same index category can be computed pointwise.

\subsection{Group and ring objects in pro-sets} \label{sec:group-objects}
In this paper we consider only commutative rings.
To distinguish between unital and non-unital rings we reserve the word ``ring'' only for unital rings and refer to non-unital rings as {\it rngs}.
We denote the category of rings (resp. rngs) as $\textbf{Ring}$ (resp. $\textbf{Rng}$).

Let $T$ be an algebraic theory. Throughout this paper we will be mostly interested in the case where $T$ is the theory of groups or the theory of rngs.

Since the category $\Pro(\Set)$ is cartesian monoidal, we may speak of the category $\Mod(T, \Pro(\Set))$ of models of $T$ in \(\Pro(\Set)\).
In the special case where $T$ is the theory of groups (resp. rngs) this category is precisely the category of group (resp. rng) objects in $\Pro(\Set)$.

The forgetful functor $F$ from the category of pro-groups \(\Pro(\Group)\) to the category of pro-sets \(\Pro(\Set)\) is faithful. Every pro-group $G^{(\infty)}$ defines the structure of a group object in \(\Pro(\Set)\) on $F\left(G^{(\infty)}\right)$. It is easy to see that a morphism \(f \in \Pro(\Set)\left(F\left(G^{(\infty)}\right), F\left(H^{(\infty)}\right)\right)\) comes from \(\Pro(\Group)\) if and only if it is a morphism of group objects. Indeed, if \(f \colon F(G^{(\infty)}) \to F(H^{(\infty)})\) is a pre-morphism and the corresponding morphism preserves the multiplication in the sense of group objects, then the maps \(f^{(i)} \colon G^{(f^*(i))} \to H^{(i)}\) are group homomorphisms after a restriction to \(G^{(j)}\) for sufficiently large \(j\). Thus, $\Pro(\Group)$ is equivalent to a full subcategory of the category of group objects in $\Pro(\Set)$.

\begin{df} \label{df-pro-set-morphisms} 
 Let $n$ be a natural number and $F_{n, T}$ be the free algebra on $n$ generators in the theory $T$.
 If $T$ is the theory of groups, then $F_n$ is the free group $F(t_1,\ldots, t_n)$.
 Similarly, if $T$ is the theory of rngs, $F_{n, T}$ is the sub-rng of the ring $\mathbb{Z}[t_1,\ldots, t_n]$ consisting of polynomials over $\mathbb{Z}$ without free terms.
 
 Given a word $w \in F_{n, T}$ and a $T$-model $X^{(\infty)}$ in $\Pro(\Set)$, one can construct the $\Pro(\Set)$-morphism
 \[ w^{(\infty)} \colon \underbrace{X^{(\infty)} \times \ldots \times X^{(\infty)}}_{n\text{ times}} \to X^{(\infty)}, \]
 which ``interprets'' the word $w$. This morphism can be obtained as an appropriate composition of diagonal morphisms $\Delta_{X^{(\infty)}}$, projections $\pi_i$ 
 (which are part of the structure of a cartesian monoidal category on $\Pro(\Set)$) 
 and the morphisms defining the structure of a $T$-model on $X^{(\infty)}$.
 
 In order to simplify the notation, to denote the interpretation morphism $w^{(\infty)}$ we often use the original term for $w$, in which every symbol of a free variable $t_i$ is replaced by the symbol $t_i^{(\infty)}$.
\end{df}

\begin{example}\label{example-commutator}
Consider the case where $T$ is the theory of groups.
Let $G^{(\infty)}$ be a group object in $\Pro(\Set)$ and $w = [t_1, t_2] = t_1 t_2 t_1^{-1} t_2^{-1} \in F(t_1, t_2)$ be the word representing the generic commutator.
In this case the interpretation morphism $w^{(\infty)}$ (also denoted $[t_1^{(\infty)}, t_2^{(\infty)}]$) can be defined as the composition
 \[ \xymatrix{G^{(\infty)} \times G^{(\infty)} \ar[r]^(.35){\Delta \times \Delta} & G^{(\infty)} \times G^{(\infty)} \times G^{(\infty)} \times G^{(\infty)} \ar[d]_{\langle \pi_1, \pi_3, i\pi_2, i\pi_4 \rangle} & \\
    & G^{(\infty)} \times G^{(\infty)} \times G^{(\infty)} \times G^{(\infty)} \ar[r]^(.76){m(m\times m)} & G^{(\infty)},} \]
where the structure of a group object on $G^{(\infty)}$ is given by the triple $(m, i, e)$.
\end{example}

\subsection{Homotopes of rings and pro-rings} \label{sec:rng-homotopes}
For the rest of the paper $R$ denotes an arbitrary ring and \(S\) denotes a fixed multiplicative subset of $R$ containing a unit. Denote by $\mathcal{S}$ the category, whose objects are the elements of \(S\) and whose morphisms \(\mathcal{S}(s, s')\) are all \(s'' \in S\) such that \(ss'' = s'\). The composition and the identity morphisms are induced by the ring structure on $R$. It is clear that $\mathcal{S}$ is a filtered category. Unless stated otherwise, all the pro-sets that we encounter in the sequel have \(\mathcal S\) as their category of indices.

We introduce the notion of a {\it homotope} of a ring inspired by a similar notion from nonassociative algebra.
\begin{df} \label{ring-homotope}
 Let $s$ be an element of $R$.  
 By definition, the {\it \(s\)-homotope} of \(R\) is the rng \(R^{(s)} = \{a^{(s)} \mid a \in R\}\) with the operations of addition and multiplication given by
 \[ a^{(s)} + b^{(s)} = (a + b)^{(s)},\ \ a^{(s)} b^{(s)} = (asb)^{(s)},\ a, b\in R.\]
 Clearly, $R^{(s)}$ has the structure of an \(R\)-algebra given by the formula \[a \cdot b^{(s)} = (ab)^{(s)},\ a, b \in R.\] For \(s, s' \in S\) there is a homomorphism of \(R\)-algebras \[R^{(ss')} \to R^{(s')}, a^{(ss')} \mapsto (as)^{(s')}.\]
 Denote by \(R^{(\infty)}\) the formal projective limit of the projective system \(R^{(s)}\), where \(s \in Ob(\mathcal S)\).
 Thus, $R^{(\infty)}$ is an object of the category $\Pro(\Rng)$. 
 It can also be considered as an rng object in \(\Pro(\Set)\).  
\end{df}

Notice that every pro-rng $R^{(\infty)}$ can be considered as a pro-group by forgetting its multiplicative structure pointwise.

\begin{rem}\label{rem:prorings-comment}
 Denote by $sR$ the principal ideal of $R$ generated by $s$. 
 There is an rng homomorphism $R^{(s)} \to sR$ given by $r^{(s)}\mapsto sr$.
 In the case where $R$ is an integral domain this homomorphism is easily seen to be an rng isomorphism. 
 
 The reason why we use homotopes rather than principal ideals in the definition of $R^{(\infty)}$ is that
 there is always a ``division by $s$'' homomorphism of $R$-modules $R^{(ss')} \to R^{(s')}$ given by $r^{(ss')} \mapsto r^{(s')},$
 while the similar map $ss'R \to s'R$ may not exist if $R$ does not happen to be a domain.
 The existence of this division homomorphism will be important in~\cref{sec:local-action}.
 
 Notice also that the projective limit of $R^{(s)}$ in $\Rng$ is often trivial.
 Indeed, if $R$ is a domain, then the limit of $R^{(s)}$ in $\Rng$ computes the intersection of the principal ideals $\bigcap_{s\in S} sR$, which often coincides with the zero rng. 
 This also shows that the set of {\it global elements} of $R^{(\infty)}$ (i.\,e. the hom-set in $\Pro(\Set)$ from the terminal object $1$ to $R^{(\infty)}$) is often trivial and therefore will be of little interest to us.
\end{rem}

Recall that a morphism \(f \in \mathcal C(X, Y)\) is called a split epimorphism (or a retraction) if it admits a section,
 i.\,e. there exists $g \in \mathcal{C}(Y, X)$ such that $fg = \mathrm{id}_{Y}$. Retractions are preserved under pullbacks.

\begin{lemma}\label{RingGeneration}
The rng multiplication morphism $m \colon R^{(\infty)} \times R^{(\infty)} \to R^{(\infty)}$ is a split epimorphism of pro-sets.
\end{lemma}
\begin{proof}
Consider the following pre-morphism of pro-sets:
\[u \colon R^{(\infty)} \to R^{(\infty)} \times R^{(\infty)}, \enskip u^*(s) = s^2, \enskip u^{(s)} \colon c^{(s^2)} \mapsto (1^{(s)}, c^{(s)}).\]
Clearly, this is indeed a pre-morphism and
\[m^{(s)}\bigl(u^{(s)}\bigl(c^{(s^2)}\bigr)\bigr) = 1^{(s)} c^{(s)} = (sc)^{(s)},\]
which shows that $mu = \mathrm{id}_{R^{(\infty)}}$.
\end{proof}

The majority of the calculations encountered in the present paper occur in the category of pro-sets or pro-groups.
Most often we need to prove certain equalities between composite pro-set or pro-group morphisms.
It turns out that the usual notation for categorical composition makes these calculations too lengthy and hardly readable. In order to remedy this and also make our computations look like the usual computations with root unipotents in Steinberg groups, we need to introduce a certain way to denote pro-set (pro-group) morphisms, specifically composite ones.

\begin{conv} \label{conv:notation}
First of all, notice that the present conventions apply only to algebraic expressions in which the symbol $(\infty)$ occurs in the upper index of all free variables (e.\,g. the identities of~\cref{RingPresentation} or~\cref{lem:elim-lhs}, but not the identities~\eqref{R1}--\eqref{R4}). We call such expressions {\it pro-expressions}.
 \begin{itemize}
  \item Any pro-expression encountered in the sequel is meant to denote a certain morphism of pro-sets. Apart from the variables marked with the index $(\infty)$, a pro-expression may also involve group or rng operations and other pro-set morphisms.
  \item Whenever a pro-expression does not involve other pro-set morphisms, it should be understood according to~\cref{df-pro-set-morphisms} (the domain and codomain will usually be clear from the context);
  \item To denote the composition of pro-set morphisms we use the syntax of substituted expressions. For example, if $f, g \colon X^{(\infty)} \to Y^{(\infty)}$ are morphisms of pro-groups and $[t_1^{(\infty)}, t_2^{(\infty)}]$ is the pro-group morphism from~\cref{example-commutator}, then $[f(a^{(\infty)}), g(b^{(\infty)})]$ denotes the composite morphism \[[t_1^{(\infty)}, t_2^{(\infty)}] \circ (f\times g) \colon X^{(\infty)} \times X^{(\infty)} \to Y^{(\infty)}.\]  
  \item Any equality of pro-expressions should be understood as the equality of pro-set morphisms defined by these expressions.
  \item The exact names of the variables occuring in a pro-expression are unimportant and are usually chosen arbitrarily. On the other hand, there is always some natural order on the variables that allows one to read the pro-expression unequivocally. In particular, $[t_1^{(\infty)}, t_2^{(\infty)}]$ and $[a^{(\infty)}, b^{(\infty)}]$ denote the same pro-set morphism $G^{(\infty)} \times G^{(\infty)} \to G^{(\infty)}$, while $[t_1^{(\infty)}, t_2^{(\infty)}]$ and $[t_2^{(\infty)}, t_1^{(\infty)}]$ are two different morphisms.
  \item The domain of the pro-set morphism defined by a pro-expression will usually be clear from the context. 
  Most often, it is a power of a single pro-object. In this case to determine the exponent one should count the number of different variables occuring in the expression.
  \item The notation for the multiplication operation is usually suppressed, i.\,e. we prefer the notation $a^{(\infty)} b^{(\infty)}$ to $m(a^{(\infty)}, b^{(\infty)})$ or $a^{(\infty)} \cdot b^{(\infty)}$.
  \item The syntax of tuples is used to denote the product of morphisms. For example, if $f,g \colon X^{(\infty)} \to Y^{(\infty)}$ are morphisms of pro-sets, then the notation $(f(x_1^{(\infty)}), g(x_2^{(\infty)}))$ means simply $f\times g$.
  \item If $g$ is a morphism of pro-sets with $X^{(\infty)} \times Y^{\infty}$ as its domain, then we write $g(a^{(\infty)}, b^{(\infty)})$ instead of $g((a^{(\infty)}, b^{(\infty)}))$.
  \item Notice that the trivial group $1$ is a zero object in the category $\Pro(\Group)$, therefore for any pro-groups $G^{(\infty)}, H^{(\infty)}$ there is a unique morphism $G^{(\infty)} \to H^{(\infty)}$ passing through $1$, this morphism will also be denoted by $1$.
 \end{itemize}
\end{conv} 
Now we are ready to formulate our next result.
\begin{lemma}\label{RingPresentation}
Let \(G^{(\infty)}\) be a pro-group, $R^{(\infty)}$ be the pro-rng defined above and \(g \colon R^{(\infty)} \times R^{(\infty)} \to G^{(\infty)}\) be a morphism of pro-sets. There is a morphism \(f \colon R^{(\infty)} \to G^{(\infty)}\) of pro-groups such that
\[g\bigl(a^{(\infty)} , b^{(\infty)}\bigr) = f\bigl(a^{(\infty)} b^{(\infty)}\bigr)\]
if and only if \(g\) satisfies the following identities in $\Pro(\Set)$:
\begin{itemize}
\item \(\bigl[g\bigl(a_1^{(\infty)}, b_1^{(\infty)}\bigr), g\bigl(a_2^{(\infty)}, b_2^{(\infty)}\bigr)\bigr] = 1\);
\item \(g\bigl(a_1^{(\infty)} + a_2^{(\infty)}, b^{(\infty)}\bigr) = g\bigl(a_1^{(\infty)}, b^{(\infty)}\bigr)\, g\bigl(a_2^{(\infty)}, b^{(\infty)}\bigr)\);
\item \(g\bigl(a^{(\infty)}, b_1^{(\infty)} + b_2^{(\infty)}\bigr) = g\bigl(a^{(\infty)}, b_1^{(\infty)}\bigr)\, g\bigl(a^{(\infty)}, b_2^{(\infty)}\bigr)\);
\item \(g\bigl(a^{(\infty)} b^{(\infty)}, c^{(\infty)}\bigr) = g\bigl(a^{(\infty)}, b^{(\infty)} c^{(\infty)}\bigr)\).
\end{itemize}
\end{lemma}
\begin{proof}
The necessity of the identities is clear.
By~\cref{RingGeneration} the morphism \(f\) is unique, so it suffices to show that it exists.
By~\cref{lem:proobj-is-a-limit} it suffices to consider the case where \(G\) is a group. 
Let \(g\) be a morphism satisfying the above identities.
By definition, there exists $s\in S$ such that $g$ is given by a homomorphism \(g' \colon R^{(s)} \times R^{(s)} \to G\) satisfying the first three identities and the identity 
\[g'\bigl((asb)^{(s)}, c^{(s)}\bigr) = g'\bigl(a^{(s)}, (bsc)^{(s)}\bigr).\]
Consider the map \(f' \colon R^{(s^2)} \to G\) given by
\[f'\bigl(c^{(s^2)}\bigr) = g'\bigl(1^{(s)}, c^{(s)}\bigr),\]
it is a homomorphism by the first and the third identities.
From the last two identities we conclude that for all \(a, b \in R\) one has
\begin{align*}
f'\bigl(a^{(s^2)} b^{(s^2)}\bigr)
&= g' \bigl( 1^{(s)}, (s^2 ab)^{(s)} \bigr)\\
&= g' \bigl( (sa)^{(s)}, (sb)^{(s)} \bigr)\\
&= g' \bigl(a^{(s^2)}, b^{(s^2)}\bigr).
\end{align*}
It is clear that \(f'\) defines the required morphism \(f\) of pro-groups.
\end{proof}

\subsection{Steinberg groups and Steinberg pro-groups}
Let $\Phi$ be an irreducible root system of rank $\geq 3$.
We assume that the root system $\Phi$ is contained in a Euclidean space $V = \mathbb{R}^\ell$ whose inner product we denote by $(-, -)$.
For a pair of roots $\alpha, \beta \in \Phi$ we denote by $\langle \alpha, \beta \rangle$ the integer $\tfrac{2(\alpha, \beta)}{|\beta|^2}$.

Recall that a root subset $\Sigma \subseteq \Phi$ is called {\it closed} if $\alpha, \beta \in \Sigma$, $\alpha+\beta\in\Phi$ imply $\alpha+\beta\in \Sigma$. A closed root subset $\Sigma$ is called {\it symmetric} (resp. {\it special}) if $\Psi = -\Psi$ (resp. $\Psi \cap -\Psi = \emptyset$). By definition, a {\it root subsystem} $\Psi \subseteq \Phi$ is a symmetric and closed root subset.

We start by recalling the definition of the Steinberg group.
\begin{df} \label{def:Steinberg}
Let $R$ be a ring. The {\it Steinberg group $\St(\Phi, R)$} is given by generators $x_\alpha(a)$, where $\alpha \in \Phi$ and $a \in R$ and the following list of defining relations (cf.~\cite{Re75}):
\begin{align}
 x_\alpha(a) \cdot x_\alpha(b)    &= x_\alpha(a+b); \tag{R1} \label{R1} \\
 [x_\alpha(a),\ x_\beta(b)] &= 1, \tag{R2} \label{R2} \\ 
 \multispan2{\hfil if $\alpha + \beta \not\in\Phi \cup \{0\};$} \nonumber \\
 [x_\alpha(a),\ x_\beta(b)] &= x_{\alpha + \beta}(N_{\alpha,\beta} \cdot ab), \tag{R3} \label{R3} \\
 \multispan2{\hfill if $\alpha+\beta\in\Phi$ but $\alpha+2\beta,\ 2\alpha+\beta\not\in\Phi;$} \nonumber \\
 [x_\alpha(a),\ x_\beta(b)] &= x_{\alpha + \beta}(N_{\alpha,\beta} \cdot ab) \cdot x_{2\alpha+\beta}(N_{\alpha,\beta}^{2,1} \cdot a^2b), \tag{R4} \label{R4} \\ \multispan2{ \hfill if $\alpha+\beta,2\alpha+\beta\in\Phi$.} \nonumber  \end{align}
\end{df}
The constants $N_{\alpha,\beta}$ and $N_{\alpha,\beta}^{2,1}$ appearing in the above relations are called the {\it structure constants} of the Chevalley group of type $\Phi$. Let us take a closer look at them.

First of all, notice that we excluded the case $\Phi=\mathsf{G}_2$ so the only possibilities for $N_{\alpha, \beta}$ appearing in the above relations are $\pm 1$ or $\pm 2$.
Notice that $|N_{\alpha,\beta}| = 2$ only when $\alpha$ and $\beta$ are short but $\alpha+\beta$ is long, in which case we set $\widehat{N}_{\alpha, \beta} = \frac{1}{2} N_{\alpha, \beta}$.
In the other cases $|N_{\alpha, \beta}| = 1$.
Now, by definition, 
\begin{align}
\label{eq:N21def}
N_{\alpha,\beta}^{2,1} = N_{\alpha,\beta} \cdot \widehat{N}_{\alpha, \alpha+\beta}.
\end{align}
It is clear that $|N_{\alpha,\beta}^{2,1}|=1$.

Many different methods of the choice of signs of the structure constants have been proposed in the literature, see e.\,g.~\cite{VP}. 
Regardless of their concrete choice, however, the structure constants always must satisfy certain relations. First of all, recall from~\cite[\S~14]{VP} that
\begin{equation} \label{eq:sc-ids-sl} N_{\alpha, \beta} = -N_{\beta,\alpha} = - N_{-\alpha, -\beta} = \tfrac{|\alpha+\beta|^2}{|\alpha|^2} N_{\beta, -\alpha-\beta} = \tfrac{|\alpha+\beta|^2}{|\beta|^2} N_{-\alpha-\beta, \alpha}. \end{equation}
These identities will be used in the sequel without explicit reference. 

We also will need another identity for structure constants.
To formulate it succinctly, we extend the domain of the structure constant function $N_{-,-}$ by setting $N_{\alpha, \beta} = 0$ whenever $\alpha+\beta\not\in\Phi\setminus\{0\}$.
Now if $\alpha, \beta, \gamma$ is a triple of pairwise linearly independent roots such that $\alpha+\beta+\gamma\neq 0$, then one has 
\begin{equation}\label{eq:cocycle2} N_{\alpha,\beta+\gamma} N_{\beta,\gamma} + N_{\beta,\gamma+\alpha} N_{\gamma,\alpha} + N_{\gamma,\alpha+\beta}N_{\alpha,\beta} = 0. \end{equation}
This identity is an equivalent form of (N9) in~\cite[\S~14]{VP} (cf. also (H4) in~\cite{Re75}).

Sometimes it is convenient to write down the relations~\eqref{R2}--\eqref{R4} as a single relation
\begin{equation}\label{eq:R234} [x_\alpha(a),\ x_\beta(b)] = \prod\limits_{\substack{i\alpha + j\beta\in\Phi \\ i,j >0}} x_{i\alpha + j\beta}(N_{\alpha,\beta}^{i,j} \cdot a^i b^j), \end{equation}
in which $N^{1,2}_{\alpha,\beta} := -N^{2,1}_{\beta, \alpha}$ and $N^{1,1}_{\alpha,\beta}:=N_{\alpha,\beta}$.

Notice that the multiplicative identity of $R$ is actually never used in the definition of the Steinberg group, 
 which allows one to use it in the situation when $R$ is an rng.
\begin{df}\label{def:Steinberg-homotope}
 Let 
 $\Phi$ be a root system of rank $\geq 3$. Applying the Steinberg group functor $\St(\Phi, -)$ to the projective system of rngs $R^{(s)}$ (see~\cref{ring-homotope}) we obtain a projective system of groups, whose formal projective limit will be denoted by $\St^{(\infty)}(\Phi, R)$ and will be called the {\it Steinberg pro-group}. 
\end{df}

\begin{rem} \label{rem:pro-Steinberg-comment}
There is a group homomorphism $\St(\Phi, R^{(s)}) \to \St(\Phi, sR)$ given by $x_\alpha^{(s)}(a)\mapsto x_\alpha(sa)$. By~\cref{rem:prorings-comment} this homomorphism is an isomorphism in the case where $R$ is an integral domain. Thus, similarly to pro-rngs, Steinberg pro-groups often do not have global elements.
\end{rem}
 
For every root $\alpha \in \Phi$ there is a ``root subgroup'' morphism in $\Pro(\Group)$
\[x_{\alpha} \colon R^{(\infty)} \to \St^{(\infty)}(\Phi, R).\] 
defined by the pre-morphism $x_\alpha^* = \mathrm{id}_S$, $x_\alpha^{(s)}(a^{(s)}) = x_\alpha(a^{(s)})$.

We denote by $\GG_\mathrm{sc}(\Phi,-)$ the simply-connected Chevalley--Demazure group scheme corresponding to a root system $\Phi$.
We also fix a pinning for $\GG_\mathrm{sc}(\Phi,-)$. By this we mean that a split maximal torus $T$ of $\GG_\mathrm{sc}(\Phi,-)$ and a Borel subgroup $B$ containing $T$ are chosen, and for every $\alpha\in\Phi$ there is an embedding $t_\alpha\colon \mathbb{G}_a\to\GG_\mathrm{sc}(\Phi, -)$.
For $a \in R$ we call the element $t_\alpha(a)$ of $\GG_\mathrm{sc}(\Phi, R)$ an {\it elementary root unipotent}, cf.~\cite{VP}. Notice that sometimes different notation is used for these elements (e.\,g. $x_\alpha(a)$ or $e_\alpha(a)$, cf.~\cite{Ma69,St71,VP}).

Recall that for an ideal $I \trianglelefteq R$ the {\it congruence subgroup} $\GG_\mathrm{sc}(\Phi, R, I)$ is defined as 
 the kernel of the homomorphism $\GG_\mathrm{sc}(\Phi, R) \to \GG_\mathrm{sc}(\Phi, R/I)$.

Recall also that one can define the unitalization of an rng $R$ as the semidirect product $R \rtimes \ZZ$ (cf. e.\,g.~\cite[Definition~3.2]{Sin}).
 \begin{df}
Consider the projective system of congruence subgroups
 \[\GG(\Phi, R^{(s)}) := \GG_{\mathrm{sc}}(\Phi, R^{(s)} \rtimes \ZZ, R^{(s)}) = \Ker\left(\GG_{\mathrm{sc}}(\Phi, R^{(s)} \rtimes \ZZ) \to \GG_{\mathrm{sc}}(\Phi, \ZZ)\right)\] with the structure morphisms induced by the structure morphisms of \(R^{(\infty)}\). Define the {\it simply-connected Chevalley pro-group} $\GG^{(\infty)}(\Phi, R)$ as the formal projective limit of this system.
\end{df}

Recall that there is well-defined homomorphism $\mathrm{st}\colon \St(\Phi, R) \to \GG_{\mathrm{sc}}(\Phi, R)$ sending each generator $x_\alpha(a)$ to the root unipotent $t_\alpha(a)$. The pro-group analogue \(\mathrm{st} \colon \St^{(\infty)}(\Phi, R) \to \GG^{(\infty)}(\Phi, R)\) of this homomorphism is given by the pre-morphism \(\mathrm{st}^* = \mathrm{id}_S\), \(\mathrm{st}^{(s)}\bigl(x_{\alpha}(a^{(s)})\bigr) = t_\alpha(a^{(s)})\). 

\begin{rem}\label{rem:uni-rad}
Let $\Sigma \subseteq \Phi$ be a special root subset and $R$ be an rng. It is well-known that the restriction of the canonical map $\St(\Phi, R)\to\GG_{\mathrm{sc}}(\Phi, R)$ to the subgroup \[\mathrm{U}(\Sigma, R) = \langle x_{\alpha}(a) \mid a\in R,\ \alpha\in \Sigma \rangle \leq \St(\Phi, R)\] is injective. Morever, for any chosen linear order on \(\Sigma\) the map $R^{\Sigma} \to \mathrm{U}(\Sigma, R)$ given by $(r_\alpha)_\alpha \mapsto \prod_\alpha x_\alpha(r_\alpha)$ is a bijection. As a consequence, the subgroup $\mathrm{U}(\Sigma, R)$ admits presentation by means of generators $x_\alpha(a)$ for $\alpha \in \Sigma$ and the relations~\eqref{R1}--\eqref{R4} in which $\alpha, \beta \in \Sigma$.
\end{rem}

\begin{rem}
 The definition of the group $\GG(\Phi, R^{(s)})$ can be reformulated in terms of Hopf algebras. Denote by \(H_\Phi\) the Hopf \(\ZZ\)-algebra corresponding to the group scheme \(\GG_{\mathrm{sc}}(\Phi, -)\). Thus, $\GG_\mathrm{sc}(\Phi, R) = \mathbf{Ring}(H_\Phi, R)$.
 
 For a pair $A, B$ of abelian groups we define the product $A \otimeshat B$ as follows: \[A \otimeshat B := (A \otimes_{\ZZ} B) \oplus A \oplus B.\]  
 Recall that if $R_1$ and $R_2$ are rngs then $R_1 \otimeshat R_2$ realizes their coproduct in the category $\Rng$.
 The product structure is induced from that of $(R_1\rtimes\mathbb Z)\otimes_{\mathbb Z}(R_2\rtimes\mathbb Z)$. The canonical morphisms $R_i \to R_1 \otimeshat R_2$, $i=1,2$ are the obvious inclusions $r_1 \mapsto (0, r_1, 0)$, $r_2 \mapsto (0, 0, r_2)$.
 $\Rng$ has the structure of a symmetric monoidal category with zero rng as the unit object and $\otimeshat$ as the tensor product. For an rng $R$ we denote by $m$ the natural homomorphism $R^{\otimeshat 2} \to R$ given by $(s \otimes t, u, v) \mapsto st + u + v$.
 
 Notice also that that the usual $\ZZ$-module tensor product \(R_1 \otimes_{\mathbb Z} R_2\) has the natural structure of an rng and that $R_1 \otimes R_2 \subseteq R_1 \otimeshat R_2$. A priori there are no homomorphisms $R_i \to R_1 \otimes R_2$. However, if $R_1$ happens to be a ring, there is a canonical homomorphism $R_2 \to R_1 \otimes R_2$.

 Recall that a {\it Hopf monoid} $H$ in a symmetric monoidal category, by definition, is a bimonoid $(H, m, u, \Delta, \varepsilon)$ such that there is a morphism $\varsigma\colon H\to H$, called the {\it antipode}, satisfying $u \circ \varepsilon = m \circ (1 \otimes \varsigma) \circ \Delta$.
 We claim that the augmentation ideal \(\Ker(\varepsilon)\) of $H_\Phi$ has the structure of a Hopf monoid in \(\Rng\). The coproduct $\Delta$ and the antipode $\varsigma$ are induced from the corresponding morphisms defining the structure of a Hopf algebra on $H_\Phi$. The counit $\varepsilon$ is the unique homomorphism $\Ker(\varepsilon) \to 0$.
 
 Now if $R$ is an arbitary ring and $s \in R$ then there is a canonical isomorphism \(\Rng(\Ker(\varepsilon), R^{(s)}) \cong \GG(\Phi, R^{(s)})\) induced by the mapping 
 \[(h \colon \Ker(\varepsilon) \to R^{(s)}) \mapsto (\widetilde{h} \colon H_\Phi \to R^{(s)}\rtimes\ZZ), \]
 where $\widetilde{h}(a) = \left(h(a - \varepsilon(a)),\, \varepsilon(a)\right)$.
 The structure of a group on $\GG(\Phi, R^{(s)})$ is then induced from the Hopf monoid structure on $\Ker(\varepsilon)$ by duality.
 For instance, the neutral element in this group is the zero map \(0 \colon \Ker(\varepsilon) \to R^{(s)}\). The product of \(g, h \in \Rng(\Ker(\varepsilon), R^{(s)})\) is the composite map
 \[ \xymatrix{ \Ker(\varepsilon) \ar[r]^(0.45){\Delta} & \Ker(\varepsilon) ^{\otimeshat 2} \ar[r]^(0.47){g \otimeshat h} & R^{(s)} \otimeshat R^{(s)} \ar[r]^(0.6){m} & R^{(s).}} \]
\end{rem}

\begin{rem}
 To formulate the next lemma, we need to introduce a way of reinterpreting the usual Steinberg relations~\eqref{R2}--\eqref{R4} as relations between the morphisms $x_\alpha$ in the category $\Pro(\Group)$. We call these relations the {\it pro-analogues} of~\eqref{R2}--\eqref{R4}. To obtain these relations one has to add the upper index $(\infty)$ to the free variables $a$, $b$ occuring in~\eqref{R2}--\eqref{R4} and then read the resulting expressions according to~\cref{conv:notation}. The fact that these relations hold is an easy consequence of our definitions.
 
 Notice that the pro-analogue of~\eqref{R1} is the identity $x_\alpha(a^{(\infty)}) x_\alpha(b^{(\infty)}) = x_\alpha(a^{(\infty)} + b^{(\infty)})$, or in the usual notation $m \circ (x_\alpha \times x_\alpha) = x_\alpha \circ +_{R^{(\infty)}}$. It is clear that this relation simply means that $x_\alpha$ is a morphism of group objects in $\Pro(\Set)$, which is a consequence of the condition that $x_\alpha$ is a pro-group morphism.
\end{rem}

Let $\mathcal{X} = \{ x_i \}_{i \in I}$ be a collection of morphisms in $\Pro(\Group)$ with a common codomain, which we denote by $S^{(\infty)}$. We say that $\mathcal{X}$ {\it generates} $S^{(\infty)}$ if for every pair of pro-group morphisms $f_1, f_2 \colon S^{(\infty)} \to G^{(\infty)}$ in order to verify the equality $f_1 = f_2$ it is enough to check that $f_1 x_i = f_2 x_i$ for $x_i \in \mathcal{X}$. 

\begin{lemma}\label{SteinbergPresentation}
The morphisms $x_{\alpha} \colon R^{(\infty)} \to \St^{(\infty)}(\Phi, R), \ \alpha\in \Phi$ generate $\St^{(\infty)}(\Phi, R)$. Moreover, for every pro-group $G^{(\infty)}$ to obtain a morphism $f \colon \St^{(\infty)}(\Phi, R) \to G$ it suffices to construct a collection of pro-group morphisms \(f_{\alpha} \colon R^{(\infty)} \to G\), $\alpha\in\Phi$, satisfying the pro-analogues of~\eqref{R2}--\eqref{R4}, in which $x_\alpha$'s are replaced with $f_{\alpha}$'s.
\end{lemma}
\begin{proof}
Notice that by~\cref{lem:proobj-is-a-limit} both of the assertions can be verified in the special case where $G^{(\infty)} = G$ is a group.

Let us verify the first assertion. Since there is only a finite number of roots in $\Phi$ one can find a sufficiently large index $s \in S$ such that the homomorphisms $f_1^{(s)}, f_2^{(s)}\colon \St(\Phi, R^{(s)}) \to G$ become equal after precomposition with $x_\alpha$ for all $\alpha \in \Phi$.
Since the root elements $x_\alpha(a^{(s)})$ generate $\St(\Phi, R^{(s)})$ we conclude that the morphisms $f_1$ and $f_2$ are equal.

Let us verify the second assertion. By definition, a morphism $f_\alpha \colon R^{(\infty)} \to G$ corresponds to a single group homomorphism 
 $f'_\alpha \colon R^{(s_\alpha)} \to G$ for some $s_\alpha \in S$. Precomposing each $f'_\alpha$ with the structure morhism $R^{(s)} \to R^{(s_\alpha)}$ for sufficiently large $s \in S$, we obtain a collection of homomorphisms $\widetilde{f}_\alpha \colon R^{(s)} \to G$.
 
Notice that~\eqref{R2}--\eqref{R4} specify only a {\it finite} collection of identities in $\Pro(\Set)$ to which $f_\alpha$'s must satisfy.
Unwinding the definitions, we find a sufficiently large index $s \in S$ such that $\widetilde{f}_\alpha$ satisfy the same identities (with $f_\alpha$'s replaced by $\widetilde{f}_\alpha$'s).
Thus, we obtain a group homomorphism $\St(\Phi, R^{(s)}) \to G$, which, in turn, determines a morphism $\St^{(\infty)}(\Phi, R) \to G$.
\end{proof}

\section{Elimination of roots}
Throughout this section we always assume that \(R\) is an arbitrary commutative ring,
 \(S \subseteq R\) is a multiplicative subset of $R$ and \(\Phi\) is a root system of rank \(\geq 3\) different from \(\rB_\ell\) and \(\rC_\ell\). 

\begin{lemma}\label{ThreeRoots}
Let \(\alpha, \beta, \gamma \in \Phi\) be a triple of roots having the same length such that $\alpha = \beta + \gamma$.  
Then there exist roots \(\beta_1, \gamma_1 \in \Phi \setminus (\ZZ\beta + \ZZ\gamma)\) having the same length as $\beta$ and, moreover, such that \(\beta_1 + \gamma_1 = \beta\). The roots $\beta, \gamma, \beta_1, \gamma_1$ are contained in a root subsystem $\Phi_0 \subseteq \Phi$, whose type is either \(\rA_3\) or \(\rC_3\). \end{lemma}
\begin{proof}
Denote by $\Phi'$ the subset of $\Phi$ consisting of roots having the same length as $\beta$.
It follows from our assumptions that either $\Phi$ is simply-laced and $\Phi' = \Phi$ or $\Phi = \rF_4$, in which case $\Phi'$ is a root subset isomorphic to $\rD_4$ (notice that $\Phi'$ is closed only if $\beta$ is long). In either case the Dynkin diagram of $\Phi'$ is simply laced, connected and has at least $3$ vertices. Since all root subsystems of $\Phi'$ of type $\rA_2$ lie in the same orbit under the action of the Weyl group of $\Phi'$, we may assume $\beta$ and $\gamma$ to be basic roots corresponding to adjacent nodes on the Dynkin diagram of $\Phi'$, in which case the first assertion is clear.

For any choice of $\beta_1$ and $\gamma_1$ notice that the smallest root subsystem $\Phi_0$ containing $\beta, \gamma, \beta_1, \gamma_1$ is irreducible and has rank $3$, which implies that it is either of type $\rC_3$ (if $\Phi = \rF_4$ and $\beta$ is short) or $\rA_3$ (otherwise). \end{proof}

As an immediate application of the above lemma we prove the following generation result, which is a stronger version of the first claim of~\cref{SteinbergPresentation}.
 \begin{lemma}\label{DoubleRootElimination}
  Let \(\Phi_0 \subseteq \Phi\) be a rank \(2\) subsystem.
  The collection of root subgroup morphisms $x_\gamma \colon R^{(\infty)} \to \St^{(\infty)}(\Phi, R)$ for $\gamma \in \Phi\setminus \Phi_0$ generates $\St^{(\infty)}(\Phi, R)$.
 \end{lemma}
 \begin{proof}
  Thanks to~\cref{SteinbergPresentation} it suffices to show that $f_1 x_\delta = f_2 x_\delta$ for all $\delta \in \Phi_0$.
  Let $\delta$ be a root lying in $\Phi_0$. We claim that $\delta$ can be decomposed as a sum $\gamma_1 + \gamma_2$ for some $\gamma_1,\gamma_2 \in \Phi\setminus\Phi_0$ having the same length as $\delta$.
  We need to consider several cases. 
  
  In the case \(\Phi_0 \cong \rA_2\) the claim follows from~\cref{ThreeRoots}.
  Consider the case $\Phi_0 \cong \rB_2$, which is only possible for $\Phi=\rF_4$.
  Since all root subsystems of type $\rB_2$ lie in the same orbit under the action of $W(\rF_4)$ (cf. e.\,g.~\cite[Table~8]{Ca72}), we may assume that $\Phi_0$ corresponds to the central edge of the Dynkin diagram of $\rF_4$ and $\delta$ is one of the nodes incident to it, in which case the claim is obvious. Finally, consider the case $\Phi_0 \cong \rA_1+\rA_1$. If either $\Phi$ is simply-laced or $\Phi_0$ has roots of different length, acting by a suitable element of $W(\Phi)$ we may achieve that the basis vectors of $\Phi_0$ correspond to a pair of unjoined nodes of the Dynkin diagram of $\Phi$, in which case the claim is also obvious. On the other hand, if $\Phi=\rF_4$ and the roots of $\Phi_0$ have the same length, the claim can be proved by passing to the root subsystem $\Phi'\cong \rD_4$ consisting of roots of $\Phi$ having the same length as the roots of $\Phi_0$.    
  
  As an immediate consequence of the claim we conclude that
  \begin{multline*} f_1(x_\delta(a^{(\infty)}b^{(\infty)})) = [f_1(x_{\gamma_1}(N_{\gamma_1, \gamma_2}a^{(\infty)})),\ f_1(x_{\gamma_2}(b^{(\infty)}))] = \\
  = [f_2(x_{\gamma_1}(N_{\gamma_1, \gamma_2}a^{(\infty)})),\ f_2(x_{\gamma_2}(b^{(\infty)}))] = f_2(x_\delta(a^{(\infty)}b^{(\infty)})), \end{multline*}
  which together with~\cref{RingGeneration} implies the required equality $f_1 x_\delta = f_2 x_\delta$.
 \end{proof}

\begin{df}
We denote by $\St(\Phi\setminus\{\alpha\}, R)$ the group given by generators $x_\beta(a)$, $\beta \neq \alpha$, $a\in R$ and the subset of the set of relations \eqref{R1}--\eqref{R4} consisting of those relations, expressions for which do not contain $x_{\alpha}(a)$. 

Similarly to~\cref{def:Steinberg-homotope} one can define the pro-group $\St^{(\infty)}(\Phi \setminus\{ \alpha\}, R)$ as
the formal projective limit of the projective system of groups $\St(\Phi\setminus\{ \alpha\}, R^{(s)})$ where $s\in S$.
\end{df}

For $\beta \neq \alpha$ we use the same notation $x_\beta$ for the ``root subgroup'' morphisms 
$R^{(\infty)}\to \St^{(\infty)}(\Phi \setminus\{\alpha\}, R).$
Obviously, these morphisms satisfy the pro-analogues of~\eqref{R2}--\eqref{R4} not containing $\alpha$ in their notations.

We can formulate the analogue of~\cref{SteinbergPresentation} for $\St^{(\infty)}(\Phi \setminus \{ \alpha \}, R)$.
\begin{lemma}\label{rem:xpma-presentation}
The morphisms $x_\beta$ for $\beta\neq \alpha$ generate $\St^{(\infty)}(\Phi\setminus\{\alpha\}, R)$. To construct a morphism $f \colon \St^{(\infty)}(\Phi\setminus\{\alpha\}, R) \to G^{(\infty)}$ it suffices to construct a collection of pro-group morphisms $f_\beta \colon R^{(\infty)} \to G^{(\infty)}$ satisfying the pro-analogues of those relations \eqref{R2}--\eqref{R4} in which $x_\alpha$ does not appear (with $x_\beta$'s replaced with $f_\beta$'s). Moreover, among these pro-analogues we can omit all relations of the form $[x_\beta(b^{(\infty)}),\ x_\gamma(c^{(\infty)})] = 1$ for all pairs of long roots \(\beta, \gamma\) such that \(\beta + \gamma = 2\alpha\).
\end{lemma}
\begin{proof}
The proof of the first two assertions is analogous to~\cref{SteinbergPresentation}. Let us verify the last assertion. Notice that \(\Phi\) must be of type $\rF_4$, otherwise there are no such relations. Acting by a suitable element of $W(\rF_4)$ we may assume, without loss of generality, that \(\alpha = \frac{1}{2}(\mathrm e_1 + \mathrm e_2 + \mathrm e_3 + \mathrm e_4)\), \(\beta = \mathrm e_1 + \mathrm e_2\), \(\gamma = \mathrm e_3 + \mathrm e_4\) (we assume that $\rF_4$ is realized in $\mathbb{R}^4$ as in~\cite[Ch.~VI,\S~4.9]{Bou81}). Set \(\gamma_1 = \mathrm e_4\) and \(\gamma_2 = \mathrm e_3 - \mathrm e_4\). Since \(\beta + \gamma_1, \beta + \gamma_2, \beta + \gamma_1 + \gamma_2 \notin \Phi \cup \{0, 2\alpha\}\)
\begin{align*}
\bigl[x_\beta(b^{(\infty)}),\ x_\gamma(N_{\gamma_1, \gamma_2}^{2, 1} {c^{(\infty)}}^2 d^{(\infty)})\bigr] \\ = \bigl[x_\beta(b^{(\infty)}),\ x_{\gamma_1 + \gamma_2}(N_{\gamma_1, \gamma_2} c^{(\infty)} d^{(\infty)})\, x_\gamma(N_{\gamma_1, \gamma_2}^{2, 1} {c^{(\infty)}}^2 d^{(\infty)})\bigr] & \text{ by~\eqref{R2},\eqref{eq:comm-mult-rhs}}\\
= \bigl[x_\beta(b^{(\infty)}), \bigl[x_{\gamma_1}(c^{(\infty)}), x_{\gamma_2}(d^{(\infty)})\bigr]\bigr] = 1 &\text{ by~\eqref{R2},\eqref{R3},\eqref{eq:HW-corr}}.
\end{align*}
The morphism ${c^{(\infty)}}^2 d^{(\infty)} \colon R^{(\infty)} \times R^{(\infty)} \to R^{(\infty)}$ is a split epimorphism of pro-sets by an argument similar to the proof of~\cref{RingGeneration}. Thus, the identity \([x_\beta(b^{(\infty)}), x_\gamma(c^{(\infty)})] = 1\) is a consequence of the other defining relations for $\St^{(\infty)}(\Phi\setminus\{\alpha\}, R)$.
\end{proof}

There is a morphism of pro-groups
$F_\alpha \colon \St^{(\infty)}(\Phi \setminus\{\alpha\}, R) \to \St^{(\infty)}(\Phi, R).$
given by the pre-morphism $F_\alpha^{*} = \mathrm{id}_S$ and the homomorphisms \[F_\alpha^{(s)} \colon \St(\Phi \setminus\{\alpha\}, R^{(s)}) \to \St(\Phi, R^{(s)})\] induced by the obvious embedding of generators.

Notice that, in general, the individual homomorphisms $F_\alpha^{(s)}$ need {\it not} be isomorphisms (at least if $s \neq 1$). On the other hand, as the following result shows, the morphism of pro-groups $F_\alpha$ is an isomorphism. 
\begin{theorem}\label{SingleRootElimination} For every root \(\alpha \in \Phi\) the morphism $F_\alpha$ is an isomorphism of pro-groups. \end{theorem}
The proof of~\cref{SingleRootElimination} occupies the rest of this subsection.
Our immediate goal is to construct the root subgroup morphism
\[\widetilde x_{ \alpha} \colon R^{(\infty)} \to \St^{(\infty)}(\Phi \setminus \{ \alpha\}, R) \] ``missing'' from the presentation of $\St^{(\infty)}(\Phi\setminus\{\alpha\}, R)$. 

\begin{df} \label{def:xbg}
Let $\beta, \gamma \in \Phi$ be an arbitrary fixed pair of roots such that $\alpha = \beta+\gamma$ and $\alpha$, $\beta$, and $\gamma$ have the same length. 
Denote by $x_{\beta,\gamma}(b^{(\infty)}, c^{(\infty)})$ the morphism
 $[x_\beta(N_{\beta,\gamma}b^{(\infty)}),\ x_\gamma(c^{(\infty)})]\colon R^{(\infty)} \times R^{(\infty)} \to \St^{(\infty)}(\Phi \setminus\{\alpha\}, R).$
\end{df}
Below we will construct the morphism $\widetilde x_\alpha$ based on the morphism $x_{\beta,\gamma}$, see~\cref{lem:new-root}. However, in order to be able to do this we first need to prove the following assertion.

\begin{lemma} \label{lem:elim-lhs}
 The morphism $x_{\beta, \gamma}$ satisfies the following relations:
 \begin{multline} \label{R2-bg}
  [x_{\beta, \gamma}(b^{(\infty)}, c^{(\infty)}),\ x_\delta(d^{(\infty)})] = 1, \hfill \text{if $\alpha + \delta \not\in\Phi \cup \{0\},\ \delta\neq\alpha;$}
 \end{multline}
 \begin{multline} \label{R3-bg}
   [x_{\beta, \gamma}(b^{(\infty)}, c^{(\infty)}),\ x_\delta(d^{(\infty)})] = x_{\alpha + \delta}(N_{\alpha,\delta} \cdot b^{(\infty)}c^{(\infty)}d^{(\infty)}), \hfill \\ \text{if $\alpha+\delta\in\Phi$ but $\alpha+2\delta,\ 2\alpha+\delta\not\in\Phi;$}
 \end{multline}
 {\setlength{\abovedisplayskip}{0pt} \setlength{\belowdisplayskip}{0pt}
 \begin{multline} \label{R4-bg1}
  [x_{\beta, \gamma}(b^{(\infty)}, c^{(\infty)}),\ x_\delta(d^{(\infty)})] = \\ = x_{\alpha + \delta}(N_{\alpha,\delta} \cdot b^{(\infty)}c^{(\infty)}d^{(\infty)}) \cdot  x_{2\alpha+\delta}(N_{\alpha,\delta}^{2,1} \cdot {b^{(\infty)}}^2{c^{(\infty)}}^2d^{(\infty)}) \\
  \text{if $\alpha+\delta,2\alpha+\delta\in\Phi.$}
 \end{multline} 
 \begin{multline} \label{R4-bg2}
  [x_\delta(d^{(\infty)}), x_{\beta, \gamma}(b^{(\infty)}, c^{(\infty)})] = \\ = x_{\alpha + \delta}(N_{\delta, \alpha} \cdot b^{(\infty)}c^{(\infty)}d^{(\infty)}) \cdot x_{\alpha+2\delta}(N_{\delta, \alpha}^{2,1} \cdot {b^{(\infty)}}{c^{(\infty)}}{d^{(\infty)}}^2) \\
  \text{if $\alpha+\delta,\alpha+2\delta\in\Phi.$}
 \end{multline}}
\end{lemma}
\begin{proof}
 First of all, observe that the above relations are obtained from \eqref{R2}--\eqref{R4} by replacing $x_\alpha(a^{(\infty)})$ with $x_{\beta, \gamma}(b^{\infty}, c^{\infty})$.
 
 Let us first consider the case \(\delta \in \Phi \setminus (\ZZ \beta + \ZZ \gamma)\).
 We claim that the relations between the remaining root subgroup morphisms of $\St^{(\infty)}(\Phi\setminus\{\alpha\}, R)$ suffice to rewrite the left-hand side of any of the relations~\eqref{R2-bg}--\eqref{R4-bg2} as follows:
 \begin{align}
  &\bigl[x_{\beta, \gamma}(b^{(\infty)}, c^{(\infty)}), x_\delta(d^{(\infty)})\bigr] \nonumber \\
  &= \up{x_\beta(N_{\beta, \gamma} b^{(\infty)})
   x_\gamma(c^{(\infty)})
   x_\beta(-N_{\beta, \gamma} b^{(\infty)})
   x_\gamma(-c^{(\infty)})}
  {x_\delta(d^{(\infty)})}\,
  x_\delta(-d^{(\infty)}) \label{eq:elim-lhs_1} \\
  &= \prod_{\substack{i\beta + j\gamma + k\delta \in \Phi\\ i, j \geq 0; k > 0}}
  x_{i\beta + j\gamma + k\delta} \bigl(A_{i, j, k} {b^{(\infty)}}^i {c^{(\infty)}}^j {d^{(\infty)}}^k\bigr). \nonumber
 \end{align}
 Indeed, observe that the root subset $\Sigma = \Phi \cap (\ZZ_{\geq 0}\beta + \ZZ_{\geq 0}\gamma + \ZZ_{> 0}\delta)$ is special and does not contain $-\beta$, $-\gamma$ or $\alpha$. Thus, we can iteratively simplify the expression for the conjugate of $x_\delta$ in the above formula obtaining a product of root subgroup morphisms corresponding to a subset of roots of $\Sigma$ at each step.
 The integers \(A_{i, j, k}\) in~\eqref{eq:elim-lhs_1} depend only on the roots \(\beta\), \(\gamma\), \(\delta\) (a priori they also depend on the chosen order of factors). To determine them one can compute the commutator $[x_\alpha(1),\ x_\delta(1)]$ in $\St(\Phi, \ZZ)$ via the same procedure as in~\eqref{eq:elim-lhs_1}:
 \begin{equation} \label{eq:elim-lhs_2}
  [x_\alpha(1), x_\delta(1)] = \prod_{\substack{i\beta + j\gamma + k\delta \in \Phi\\ i, j \geq 0; k > 0}}
  x_{i\beta + j\gamma + k\delta}(A_{i, j, k}).
 \end{equation}
  By~\cref{rem:uni-rad} the integers $A_{i,j,k}$ are uniquely determined by~\eqref{eq:elim-lhs_2}.
  It is easy to see that, depending on the relative position of $\alpha$ and $\delta$, the integers $A_{i,j,k}$ coincide with the structure constants in the right-hand sides of~\eqref{R2},~\eqref{R3} or~\eqref{R4}, in particular, $A_{i,j,k}=0$ for $i\neq j$.
  Thus, \eqref{R2-bg}--\eqref{R4-bg2} follow from~\eqref{eq:elim-lhs_1} in the specified case.
 
 It remains to verify Steinberg relations in the case \(\delta \in \Phi \cap (\ZZ \beta + \ZZ \gamma)\).
 By~\cref{ThreeRoots} there exist roots \(\beta_1, \beta_2 \in \Phi \setminus (\ZZ \beta + \ZZ \gamma)\) such that \(|\beta_1| = |\beta_2| = |\beta|\) and \(\beta = \beta_1 + \beta_2\). 
 Recall that the root subsystem $\Phi_0 \subseteq \Phi$ containing $\beta,\gamma,\beta_1,\beta_2$ is of type $\rA_3$ or $\rC_3$. 
 In both cases we may assume that $\alpha + \beta_2 \not\in\Phi$.
 
 Let us first verify the relation~\eqref{R2-bg} in the case $\delta \in \{\beta, \gamma\}$. By symmetry we may assume that $\delta = \beta$. 
 Direct computation shows that
 \begin{align}
  [x_{\beta,\gamma}(b^{(\infty)}, c^{(\infty)}),\ x_\beta(N_{\beta_1, \beta_2} b_1^{(\infty)}b_2^{(\infty)})] \nonumber \\
   = [x_{\beta,\gamma}(b^{(\infty)}, c^{(\infty)}),\ [x_{\beta_1}(b_1^{(\infty)}),\ x_{\beta_2}(b_2^{(\infty)})]]  & \text{ by~\eqref{R3}} \label{eq:elim-lhs3} \\ 
   = [[x_{\beta, \gamma}(b^{(\infty)}, c^{(\infty)}),\ x_{\beta_1}(b_1)],\ \up{x_{\beta_1}(b_1^{(\infty)})} x_{\beta_2}(b_2^{(\infty)})]  & \text{ by~\eqref{eq:HW-corr},\eqref{R2-bg}.} \nonumber \end{align}
In the case $\Phi_0\cong \rA_3$ the inner commutator in the last expression is trivial by~\eqref{R2-bg},
 therefore so is the outer commutator (notice that $\alpha$ and $\beta_1$ form an acute angle).
In the case $\Phi_0\cong \rC_3$ the last expression can be further simplified using~\eqref{R3} and the already proved relation~\eqref{R3-bg} as follows:
\[ \ldots = [x_{\alpha+\beta_1}(N_{\alpha, \beta_1} b^{(\infty)}c^{(\infty)}b_1^{(\infty)}),\ x_{\beta}(N_{\beta_1, \beta_2} b_1^{(\infty)} b_2^{(\infty)}) x_{\beta_2}(b_2^{(\infty)})]. \]
Since $\alpha+\beta_1$ forms an acute angle with both $\beta$ and $\beta_2$, the latter commutator is trivial by~\eqref{eq:comm-mult-rhs} and~\eqref{R2-bg}.  
  In both cases the right-hand side of~\eqref{eq:elim-lhs3} is trivial. Consequently, from~\cref{RingGeneration} we obtain the equality
 $[x_{\beta,\gamma}(b^{(\infty)}, c^{(\infty)}),\ x_\beta(d^{(\infty)})] = 1.$
 Thus, we have verified~\eqref{R2-bg} in all cases. 
 
 Finally, it remains to verify~\eqref{R3-bg} in the case $\delta \in \{-\beta, -\gamma\}$.
 By symmetry it suffices to consider the case $\delta = -\beta$.
 From the already proved relations~\eqref{R2-bg}--\eqref{R3-bg} and the other Steinberg relations we obtain
 \begin{align*}
  &[x_{\beta, \gamma}(b^{(\infty)}, c^{(\infty)}),
  x_{-\beta}(N_{-\beta_1, -\beta_2} b_1^{(\infty)} b_2^{(\infty)})] \\
  &= [\up{x_{\beta, \gamma}(b^{(\infty)}, c^{(\infty)})}
   {x_{-\beta_1}(b_1^{(\infty)})},
  \up{x_{\beta, \gamma}(b^{(\infty)}, c^{(\infty)})}
   {x_{-\beta_2}(b_2^{(\infty)})}]\,
  x_{-\beta}(-N_{-\beta_1, -\beta_2} b_1^{(\infty)} b_2^{(\infty)})\\
  &= \prod_{\substack{i\alpha - j\beta_1 - k\beta_2 \in \Phi\\ i, j, k \geq 0}}
  x_{i\alpha - j\beta_1 - k\beta_2}\bigl(B_{i, j, k} {b^{(\infty)}}^i {c^{(\infty)}}^i {b_1^{(\infty)}}^j {b_2^{(\infty)}}^k\bigr),  
 \end{align*}
where $B_{i,j,k}$ are certain integers.
Direct computation (or an argument similar to the one used in the first part of the proof) shows that the constants \(B_{i, j, k}\) are all zero with the sole exception of \(B_{1, 1, 1}$, which is equal to $N_{\alpha, -\beta}\). Consequently, from~\cref{RingGeneration} we obtain
 \[[x_{\beta, \gamma}(b^{(\infty)}, c^{(\infty)}), x_{-\beta}(d^{(\infty)})] = x_\gamma(N_{\alpha, -\beta} b^{(\infty)} c^{(\infty)} d^{(\infty)}),\]
which finishes the proof of~\eqref{R3-bg}.
\end{proof}

\begin{lemma}\label{lem:new-root}
 The morphism $x_{\beta, \gamma}$ satisfies the identities listed in~\cref{RingPresentation} and therefore gives rise to a pro-group morphism $R^{(\infty)} \to \St^{(\infty)}(\Phi\setminus\{\alpha\}, R),$ which we denote by $\widetilde{x}_\alpha$. The resulting morphism $\widetilde{x}_\alpha$ does not depend on the choice of $\beta$ and $\gamma$.
\end{lemma}
\begin{proof}
 First of all, observe that~\eqref{R2-bg} implies the identity \[[x_{\beta, \gamma}(b_1^{(\infty)}, c_1^{(\infty)}),
 x_{\beta, \gamma}(b_2^{(\infty)}, c_2^{(\infty)})] = 1.\] 
 Notice also that~\eqref{eq:comm-mult-lhs} and~\eqref{R2-bg} imply that
 \begin{align*}
 x_{\beta, \gamma}\bigl(b_1^{(\infty)} + b_2^{(\infty)}, c^{(\infty)}\bigr)
 = \bigl[x_\beta(N_{\beta, \gamma} b_1^{(\infty)})
 x_\beta(N_{\beta, \gamma} b_2^{(\infty)}),
 x_\gamma(c^{(\infty)})\bigr]\\
 = x_{\beta, \gamma}(b_1^{(\infty)}, c^{(\infty)})\,
 x_{\beta, \gamma}(b_2^{(\infty)}, c^{(\infty)}).
 \end{align*}
 Similarly, one can show the equality
 \[x_{\beta, \gamma}\bigl(b^{(\infty)}, c_1^{(\infty)} + c_2^{(\infty)}\bigr)
 = x_{\beta, \gamma}(b^{(\infty)}, c_1^{(\infty)})\,
 x_{\beta, \gamma}(b^{(\infty)}, c_2^{(\infty)}).\]
 Thus, the morphism \(x_{\beta, \gamma}\) satisfies the first three requirements of~\cref{RingPresentation}. 
 
 Now suppose that \(\alpha = \beta_1 + \gamma_1 = \beta_2 + \gamma_2\) are two different decompositions for $\alpha$ such that $\alpha$, $\beta_i$ and $\gamma_i$ have the same length and \(\alpha\), \(\beta_1\), \(\beta_2\) are linearly independent. By~\cref{ThreeRoots} such decompositions exist and the roots $\beta_i$, $\gamma_i$ are contained in a root subsystem $\Phi_0$ of type \(\rA_3\) or \(\rC_3\). Swapping $\beta_1$ with $\gamma_1$, if necessary, we also may assume that $\beta_1$ and $\beta_2$ form an acute angle. Set $\delta = \gamma_1 - \gamma_2 = \beta_2 - \beta_1$ and \[\varepsilon_1 = N_{\beta_1, \gamma_1},\ \varepsilon_2 = N_{\beta_2, \gamma_2},\ \varepsilon_3 = N_{\delta,\gamma_2},\ \varepsilon_4 = N_{\beta_1, \delta}.\]
 
In order to simplify keeping track of the relative angles and the length of the roots, the reader may assume that  $\beta_i$, $\gamma_i$ are concretely realized in the space $\mathbb{R}^3$ as follows:
\(\beta_1 = \mathrm e_1 - \mathrm e_2,\) \(\gamma_1 = \mathrm e_2 - \mathrm e_3\) and
 \begin{itemize}
 \item $\beta_2 = \mathrm e_1 - \mathrm e_4$, $\gamma_2 = \mathrm e_4 - \mathrm e_3$, $\delta = \mathrm e_2 - \mathrm e_4$ in the case \(\Phi_0\cong\rA_3\), \item $\beta_2 = \mathrm e_1 + \mathrm e_2$, $\gamma_2 = -\mathrm e_2 - \mathrm e_3$, $\delta = 2 \mathrm e_2$ in the case \(\Phi_0\cong\rC_3\).
\end{itemize}
Here $\mathrm e_i$ denote the standard basis vectors of $\mathbb{R}^3$ and we assume that the root systems $\rA_3$ and $\rC_3$ are realized in $\mathbb{R}^3$ as in~\cite[Ch.~VI,~\S\S~4.6--4.7]{Bou81}. 

 Substituting the triple $(\beta_1, \gamma_1, -\beta_2)$ into~\eqref{eq:cocycle2}, it is not hard to conclude that in both cases one has $\varepsilon_1 \varepsilon_2 \varepsilon_3 \varepsilon_4 = 1$. Now direct calculation shows that
 \begin{align*}
x_{\beta_1, \gamma_1}(b^{(\infty)}, c^{(\infty)} d^{(\infty)}) = [x_{\beta_1}(\varepsilon_1b^{(\infty)}), x_{\gamma_1}(c^{(\infty)} d^{(\infty)})] & \nonumber \\
= [x_{\beta_1}(\varepsilon_1b^{(\infty)}), [x_{\delta}(\varepsilon_3c^{(\infty)}), x_{\gamma_2}(d^{(\infty)})]]&\\
\text{  by~\eqref{R3} if } \Phi_0 \cong \rA_3 \text{ and by ~\eqref{R2},\eqref{R4},\eqref{eq:comm-mult-rhs} if } \Phi_0 &\cong \rC_3 \nonumber \\ 
= [[x_{\beta_1}(\varepsilon_1b^{(\infty)}),\ x_\delta(\varepsilon_3c^{(\infty)})],\ {}^{x_\delta(\varepsilon_3c^{(\infty)})}x_{\gamma_2}(d^{(\infty)})] & \text{ by~\eqref{R2},\eqref{eq:HW-corr}}\\
= [x_{\beta_2}(\varepsilon_2b^{(\infty)}c^{(\infty)}),\ x_{\gamma_2}(d^{(\infty)}) x_{\gamma_1}(c^{(\infty)}d^{(\infty)}) ] & \text{ by~\eqref{R3}}\\
= [x_{\beta_2}(\varepsilon_2b^{(\infty)}c^{(\infty)}),\ x_{\gamma_2}(d^{(\infty)})] &\text{ by~\eqref{R2},\eqref{eq:comm-mult-rhs}} \\
= x_{\beta_2, \gamma_2}(b^{(\infty)} c^{(\infty)}, d^{(\infty)}). & \end{align*}

 Thus, in both cases we obtain the equality
 \begin{equation}\label{eq:balance}
 x_{\beta_1, \gamma_1}(b^{(\infty)}, c^{(\infty)} d^{(\infty)}) = x_{\beta_2, \gamma_2}(b^{(\infty)} c^{(\infty)}, d^{(\infty)}).
 \end{equation}
 It follows from~\eqref{eq:balance} that 
 \[x_{\beta_1,\gamma_1}(b^{(\infty)}, u_1^{(\infty)}u_2^{(\infty)}d^{(\infty)}) = x_{\beta_1,\gamma_1}(b^{(\infty)}u_1^{(\infty)}u_2^{(\infty)},d^{(\infty)})\] therefore by~\cref{RingGeneration}
 $x_{\beta_1, \gamma_1}(b^{(\infty)}, u^{(\infty)} c^{(\infty)}) = x_{\beta_1, \gamma_1}(b^{(\infty)} u^{(\infty)}, c^{(\infty)}).$
 Thus, the last requirement of~\cref{RingPresentation} is satisfied. Consequently, there exists a unique morphism of pro-groups \(\widetilde x_\alpha \colon R^{(\infty)} \to \St^{(\infty)}(\Phi \setminus \{\alpha\}, R)\) satisfying
 \begin{equation} \label{eq:new-root} x_{\beta_1, \gamma_1}(b^{(\infty)}, c^{(\infty)}) = \widetilde x_\alpha(b^{(\infty)} c^{(\infty)}) = x_{\beta_2, \gamma_2}(b^{(\infty)}, c^{(\infty)}). \qedhere\end{equation}
\end{proof}

The next step in the proof of~\cref{SingleRootElimination} is to verify that $\widetilde{x}_\alpha$ satisfies the relations~\eqref{R2}--\eqref{R4}.
Combining \cref{lem:elim-lhs,lem:new-root} we immediately obtain the following.
\begin{corollary} \label{cor:elim-lhs}
 The root subgroup morphism $\widetilde{x}_\alpha$ satisfies the pro-analogues of Steinberg relations of type~\eqref{R2}--\eqref{R4}, in which it occurs on the left-hand side.
\end{corollary}

Thus, to finish the proof of~\cref{SingleRootElimination} it suffices to consider the case where $\widetilde{x}_\alpha$ occurs in the right-hand side of the commutator formula.
\begin{lemma} \label{lem:elim-rhs-r4}
 The root subgroup morphism $\widetilde{x}_\alpha$ satisfies the pro-analogues of Steinberg relations of type~\eqref{R4} in which it occurs on the right-hand side.
\end{lemma}
\begin{proof}
We only need to consider the case $\Phi = \rF_4$.
We need to verify the relation obtained from \begin{equation}\label{eq:elim-rhs-r4-1} [x_{\beta}(b^{(\infty)}),\ x_{\gamma}(c^{(\infty)})] = x_{\beta+\gamma}(N_{\beta,\gamma}b^{(\infty)} c^{(\infty)}) x_{2\beta+\gamma}(N_{\beta,\gamma}^{2,1}{b^{(\infty)}}^2c^{(\infty)})\end{equation} by replacing either $x_{\beta+\gamma}$ or $x_{2\beta+\gamma}$ with $\widetilde{x}_{\alpha}$.

First, let us consider the case \(\alpha = \beta + \gamma\), in which \(\alpha\), \(\beta\) are short and \(\gamma\) is long. Consider a decomposition \(\beta = \beta_1 + \beta_2\) for some short roots \(\beta_i\). The smallest root subsystem containing \(\beta_i\) and \(\gamma\) is of type \(\rC_3\), so we may assume that \(\alpha = \mathrm e_1 + \mathrm e_3\), \(\beta = \mathrm e_1 - \mathrm e_3\), \(\beta_1 = \mathrm e_1 - \mathrm e_2\), \(\beta_2 = \mathrm e_2 - \mathrm e_3\), \(\gamma = 2\mathrm e_3\). Set $\varepsilon_1 = N_{\beta_1, \beta_2},$ $\varepsilon_2 = N_{\beta_2,\gamma}^{2,1},$ $\varepsilon_3 = N_{\beta_2, \gamma},$ $\varepsilon_4 = N_{\beta_1, 2\beta_2+\gamma}^{2,1},$ $\varepsilon_5 = N_{\beta_1, 2\beta_2+\gamma},$ $\varepsilon_6 = N_{\beta_1,\beta_2+\gamma},$ $\varepsilon_7 = N_{\beta_2+\gamma, \beta}.$
Substituting triples $(\beta_1, \beta_2+\gamma, \beta_2)$, $(\beta_1, \beta_2, \gamma)$, $(\beta_1, \beta_1+\beta_2, \beta_2+\gamma)$ into~\eqref{eq:cocycle2} we obtain the equalities 
\begin{equation*} \varepsilon_2 \varepsilon_3 \varepsilon_5 = - \varepsilon_1 \varepsilon_7,\ \varepsilon_3 \varepsilon_6  = N_{\beta, \gamma}\varepsilon_1,\ 2 \varepsilon_4 \varepsilon_5 \varepsilon_7 = - \varepsilon_6 N_{\beta,\alpha},\end{equation*}
which together with~\eqref{eq:N21def} imply the equality 
\begin{equation*}  \varepsilon_2 \varepsilon_4 = \tfrac{1}{2} \varepsilon_1 \varepsilon_3 \varepsilon_5 \varepsilon_7 \cdot \varepsilon_5 \varepsilon_6 \varepsilon_7 N_{\beta,\alpha}= \tfrac{1}{2} \varepsilon_1 \varepsilon_3 \varepsilon_6 N_{\beta, \alpha} = \tfrac{1}{2} N_{\beta,\gamma} N_{\beta,\alpha}= N_{\beta, \gamma}^{2, 1}.\end{equation*}
Now a direct computation using~\eqref{R2}--\eqref{R4} and \cref{cor:elim-lhs} shows that
\begin{align*}
  \bigl[x_\gamma(c^{(\infty)}), x_\beta(\varepsilon_1 b_1^{(\infty)} b_2^{(\infty)})\bigr] \\ = \bigl[x_{\beta_1}(b_1^{(\infty)}), \up{x_\gamma(c^{(\infty)})} {x_{\beta_2}(b_2^{(\infty)})}\bigr]\cdot x_\beta(-\varepsilon_1 b_1^{(\infty)} b_2^{(\infty)}) & \text{ by~\eqref{eq:comm-new}} \\
  = \bigl[x_{\beta_1}(b_1^{(\infty)}),\ x_{2\beta_2+\gamma}(-\varepsilon_2{b_2^{(\infty)}}^2c^{(\infty)}) x_{\beta_2+\gamma}(-\varepsilon_3b_2^{(\infty)}c^{(\infty)}) x_{\beta_2}(b_2^{(\infty)})\bigr] \cdot \\ \cdot x_{\beta}(-\varepsilon_1b_1^{(\infty)}b_2^{(\infty)})
  = \bigl[x_{\beta_1}(b_1^{(\infty)}), x_{2\beta_2+\gamma}(-\varepsilon_2{b_2^{(\infty)}}^2c^{(\infty)})\bigr] \cdot \qquad\qquad \\ \cdot \up{x_{2\beta_2+\gamma}(-\varepsilon_2{b_2^{(\infty)}}^2c^{(\infty)})}{\bigl[x_{\beta_1}(b_1^{(\infty)}), x_{\beta_2+\gamma}(-\varepsilon_3b_2^{(\infty)}c^{(\infty)})\bigr]} \cdot\qquad\\
 \cdot \bigl[x_{2\beta_2+\gamma}(-\varepsilon_2{b_2^{(\infty)}}^2c^{(\infty)}) x_{\beta_2+\gamma}(-\varepsilon_3b_2^{(\infty)}c^{(\infty)}), x_{\beta}(\varepsilon_1b_1^{(\infty)}b_2^{(\infty)})\bigr] & \text{ by~\eqref{eq:comm-mult-rhs}} \\
 = x_{\beta_1 + 2\beta_2 + \gamma}(-\varepsilon_2 \varepsilon_5 b_1^{(\infty)} {b_2^{(\infty)}}^2 c^{(\infty)})\cdot x_{2\beta + \gamma}(-\varepsilon_2 \varepsilon_4 {b_1^{(\infty)}}^2 {b_2^{(\infty)}}^2 c^{(\infty)}) \cdot \qquad \\
 \cdot \widetilde x_\alpha(-\varepsilon_3 \varepsilon_6 b_1^{(\infty)} b_2^{(\infty)}c^{(\infty)}) \cdot
 x_{\beta_1 + 2\beta_2 + \gamma}(-\varepsilon_1 \varepsilon_3 \varepsilon_7 b_1^{(\infty)} {b_2^{(\infty)}}^2 c^{(\infty)}) &\text{ by~\eqref{eq:comm-mult-lhs},~\eqref{eq:new-root}} \\
 = x_{2\beta + \gamma}( - N_{\beta,\gamma}^{2,1} {b_1^{(\infty)}}^2{b_2^{(\infty)}}^2c^{(\infty)}) \cdot \widetilde{x}_{\alpha}(-N_{\beta,\gamma}\varepsilon_1 b_1^{(\infty)}b_2^{(\infty)}c^{(\infty)}).
\end{align*}
 Thus, by~\cref{RingGeneration} \[[x_\gamma(c^{(\infty)}), x_\beta(b^{(\infty)})] =
 x_{2\beta+\gamma}(-N_{\beta,\gamma}^{2,1} {b^{(\infty)}}^2 {c^{(\infty)}}) \cdot \widetilde{x}_\alpha(-N_{\beta, \gamma} b^{(\infty)} c^{(\infty)}),\] which is an equivalent form of~\eqref{eq:elim-rhs-r4-1}.
 
 Now consider the case \(\alpha = 2\beta + \gamma\), in which \(\beta\) is short and \(\alpha\), \(\gamma\) are long. We can choose a decomposition \(\gamma = \gamma_1 + \gamma_2\) for some long roots \(\gamma_i\). The smallest root subsystem containing \(\beta\) and \(\gamma_i\) is of type \(\rB_3\), so we may assume that \(\alpha = \mathrm e_1 + \mathrm e_3\), \(\beta = \mathrm e_3\), \(\gamma = \mathrm e_1 - \mathrm e_3\), \(\gamma_1 = \mathrm e_1 - \mathrm e_2\), \(\gamma_2 = \mathrm e_2 - \mathrm e_3\). Set $\varepsilon_1 = N_{\gamma_1, \gamma_2}$, $\varepsilon_2 = N_{\beta,\gamma_2}^{2,1}$, $\varepsilon_3 = N_{\beta,\gamma_2}$, $\varepsilon_4 = N_{\gamma_1, 2\beta+\gamma_2}$, $\varepsilon_5 = N_{\gamma_1, \beta+\gamma_2}$.
 Substituting the triples $(\beta, \gamma_2, \gamma_1)$, $(\beta, \beta+\gamma_2, \gamma_1)$ into~\eqref{eq:cocycle2} we obtain the identities 
 \begin{equation*}
  \varepsilon_3 \varepsilon_5 = \varepsilon_1 N_{\beta,\gamma},\ 2\varepsilon_2\varepsilon_3 \varepsilon_4 = \varepsilon_5N_{\beta,\beta+\gamma}.
 \end{equation*}
 Notice that these identities imply that $\varepsilon_2\varepsilon_4 = \tfrac{1}{2} \varepsilon_3 \varepsilon_5 N_{\beta,\beta+\gamma} = \varepsilon_1 N_{\beta,\gamma}^{2,1}$.
 A direct computation using \eqref{R2}--\eqref{R4}, \eqref{eq:comm-mult-rhs},\eqref{eq:new-root} and \cref{cor:elim-lhs} shows that
 \begin{align*}
  \bigl[x_\beta(b^{(\infty)}), x_\gamma(\varepsilon_1 c_1^{(\infty)} c_2^{(\infty)})\bigr]
   =  \bigl[x_{\gamma_1}(c_1^{(\infty)}), \up{x_\beta(b^{(\infty)})} {x_{\gamma_2}(c_2^{(\infty)})}\bigr]\,
  x_\gamma(-\varepsilon_1 c_1^{(\infty)} c_2^{(\infty)})  \\  
  = \bigl[x_{\gamma_1}(c_1^{(\infty)}), x_{2\beta+\gamma_2}(\varepsilon_2 {b^{(\infty)}}^2 c_2^{(\infty)}) x_{\beta+\gamma_2}(\varepsilon_3 b^{(\infty)} c_2^{(\infty)}) x_{\gamma_2}(c_2^{(\infty)})\bigr] x_\gamma(-\varepsilon_1 c_1^{(\infty)} c_2^{(\infty)}) \\
  = \widetilde{x}_\alpha( \varepsilon_2 \varepsilon_4 {b^{(\infty)}}^2 c_1^{(\infty)}c_2^{(\infty)}) \cdot x_{\beta+\gamma}(\varepsilon_3 \varepsilon_5 b^{(\infty)} c_1^{(\infty)}c_2^{(\infty)}). \end{align*}
Now \cref{RingGeneration} and the above identity imply~\eqref{eq:elim-rhs-r4-1}. \end{proof}

\begin{lemma} \label{lem:elim-rhs-r3}
 The root subgroup morphism $\widetilde{x}_\alpha$ satisfies the pro-analogues of Steinberg relations~of type \eqref{R3} in which it occurs on the right-hand side. 
\end{lemma}
\begin{proof}
We need to verify the relation 
\begin{equation}\label{eq:elim-rhs-r3-1} [x_{\beta}(b^{(\infty)}),\ x_{\gamma}(c^{(\infty)})] = \widetilde{x}_{\alpha}(N_{\beta,\gamma} b^{(\infty)}c^{(\infty)}).\end{equation}
In the case where $\alpha$, $\beta$ and $\gamma$ have the same length the assertion follows from \cref{lem:new-root}. Thus, we only need to consider the case where \(\alpha = \beta + \gamma\) and \(\alpha\) is long while \(\beta\) and \(\gamma\) are short. The smallest root subsystem containing \(\alpha\), \(\beta\), \(\gamma\) is of type \(\rB_2\). Without loss of generality we may assume \(\alpha = \mathrm e_1 + \mathrm e_2\), \(\beta = \mathrm e_1\), and \(\gamma = \mathrm e_2\).
Set $\varepsilon_1 = N_{\beta,\gamma-\beta}^{2,1}$, $\varepsilon_2 = N_{\beta,\gamma-\beta}$. Direct computation shows that \begin{align*}
  \up{x_\beta(d^\infty)} x_\gamma(\varepsilon_2b^{(\infty)}c^{(\infty)}) \cdot \widetilde{x}_\alpha(\varepsilon_1 {b^{(\infty)}}^2{c^{(\infty)}}) \\
  = \up{x_\beta(d^{(\infty)})}
   {\bigl(x_\gamma(\varepsilon_2 b^{(\infty)} c^{(\infty)})\,
   \widetilde x_\alpha(\varepsilon_1
   {b^{(\infty)}}^2 c^{(\infty)})\bigr)} & \text{ by \cref{cor:elim-lhs}}\\
  = \bigl[x_\beta(b^{(\infty)}),
  \up{x_\beta(d^{(\infty)})}
   {x_{\gamma - \beta}(c^{(\infty)})}\bigr] & \text{ by~\eqref{R2}, \cref{lem:elim-rhs-r4}}\\
  = \up{x_\beta(b^{(\infty)} + d^{(\infty)})}
   {x_{\gamma - \beta}(c^{(\infty)})}\cdot
  \up{x_\beta(d^{(\infty)})}
   {x_{\gamma - \beta}(-c^{(\infty)})} & \text{ by~\eqref{R1}} \\  
  = \widetilde x_\alpha(\varepsilon_1 (b^{(\infty)} + d^{(\infty)})^2 c^{(\infty)})\, x_\gamma(\varepsilon_2 (b^{(\infty)} + d^{(\infty)}) c^{(\infty)})\, x_{\gamma - \beta}(c^{(\infty)})\\
  \cdot \widetilde x_\alpha(-\varepsilon_1 c^{(\infty)} {d^{(\infty)}}^2)\, x_\gamma(-\varepsilon_2 c^{(\infty)} d^{(\infty)})\, x_{\gamma - \beta}(-c^{(\infty)})& \text{ by~\eqref{R4}}\\
  = x_{\gamma}(\varepsilon_2 b^{(\infty)}c^{(\infty)}) \widetilde{x}_{\alpha}(2\varepsilon_1{b^{(\infty)}} c^{(\infty)} {d^{(\infty)}} + \varepsilon_1 {b^{(\infty)}}^2{c^{(\infty)}})& \text{ by~\cref{lem:elim-rhs-r4}}
 \end{align*}
The relation~\eqref{eq:elim-rhs-r3-1} now follows from \cref{RingGeneration} and \eqref{eq:N21def}.
\end{proof}

\begin{proof}[Proof of~\cref{SingleRootElimination}]
To distinguish between the root subgroup morphisms of $\St^{(\infty)}(\Phi, R)$ and $\St^{(\infty)}(\Phi\setminus\{\alpha\}, R)$ we rename the root subgroup morphisms $x_\beta$, $\beta\neq \alpha$ of $\St(\Phi\setminus\{\alpha\}, R^{(\infty)})$ to $\widetilde{x}_{\beta}$. For the morphisms $x_\beta$ of $\St^{(\infty)}(\Phi, R)$ we continue to use the usual notation. 

By~\cref{SteinbergPresentation} the morphisms $\widetilde{x}_{\beta}$ give rise to a morphism $G_\alpha \colon \St^{(\infty)}(\Phi, R) \to \St^{(\infty)}(\Phi\setminus\{\alpha\}, R)$. By definition, $G_\alpha x_\beta = \widetilde{x}_\beta$ for all $\beta$. On the other hand, it is clear that $F_\alpha \widetilde{x}_\beta = x_\beta$ for $\beta\neq \alpha$. To see that $G_\alpha$ is inverse to $F_\alpha$ it remains to invoke the first part of~\cref{SteinbergPresentation} and~\cref{rem:xpma-presentation}.
\end{proof}


\section{The action of \texorpdfstring{$\GG_{\mathrm{sc}}(\Phi, S^{-1}R)$}{G(Ф, R)} on pro-groups} \label{sec:local-action}
Throughout this section \(R\) is a ring and \(S \subseteq R\) is a multiplicative subset.
\subsection{The case of Chevalley groups} \label{sec:local-Chevalley}
Our first goal is to show that the group \(\GG_{\mathrm{sc}}(\Phi, S^{-1} R)\) can be made to act on the Chevalley pro-group \(\GG^{(\infty)}(\Phi, R)\) by conjugation. Denote by \(H_\Phi\) the Hopf \(\ZZ\)-algebra of \(\GG_{\mathrm{sc}}(\Phi, -)\) with the structure of a Hopf algebra given by the triple $(\Delta, \varsigma, \varepsilon)$. 
For shortness, we denote by $K$ the augmentation ideal $\Ker(\varepsilon)$. Recall that \(H_\Phi= K \rtimes \ZZ\) is a finitely presented commutative \(\ZZ\)-algebra, therefore $K$ is a finitely generated rng. Indeed, if $x_i$ are the generators of $H_\Phi$, then $y_i:=x_i - \varepsilon(x_i)$ generate $K$ as an rng as seen e.\,g. from the identity
$$xy - \varepsilon(xy) = (x - \varepsilon(x)) (y - \varepsilon(y)) + (x - \varepsilon(x)) \varepsilon(y) + \varepsilon(x) (y - \varepsilon(y)).$$
There is a surjective homomorphism $\ZZ[X_1, \ldots, X_n] \to K \rtimes \ZZ$ sending $X_i$ to $x_i$,
whose kernel is generated by a finite collection of polynomials \(g_1, \ldots, g_m\).
Substituting $X_i \mapsto Y_i + \varepsilon(x_i)$ into $g_j$'s we obtain a new collection of polynomials $\{f_j(Y_1,\ldots Y_n)\}$ without constant terms.
It is clear that $f_j$'s form a complete system of defining relations between $y_i$. Denote by \(N\) be the maximum of the degrees of \(f_j\)'s.
Consider the ``coconjugation'' homomorphism (written down in Sweedler notation)
\[\mathrm{Coconj} \colon K \to H_\Phi \otimes K,\ h \mapsto \sum h_{(1)} \varsigma(h_{(3)}) \otimes h_{(2)}.\]

Recall that \(S^{-1} R\) is isomorphic as an $R$-module to the filtered colimit of \(R\)-modules \(\frac Rs\) over the category $\mathcal{S}$ (cf.~\cref{sec:rng-homotopes}), each of these modules being isomorphic to $R$. For $s, s' \in S$ the homomorphism \(\mathrm{can}\colon \frac R{s} \to \frac R{ss'}\) is given by $\tfrac{r}{s} \mapsto \tfrac{rs'}{ss'}$. The elements of modules $\frac Rs$ can be ``multiplied'' with each other via the natural $R$-module homomorphisms $\frac R s \otimes \frac R {s'} \to \frac R {ss'}$ compatible with the canonical homomorphisms $\frac R s \to S^{-1}R$.

Let $g$ be an element of $\GG_{\mathrm{sc}}(\Phi, S^{-1} R) \cong \mathbf{Ring}(H_\Phi, S^{-1}R)$. Consider the rng homomorphism \(g \otimes \mathrm{id} \colon H_\Phi \otimes K \to S^{-1} R \otimes K\).
Since tensor products commute with colimits, there exists \(s_0 \in S\) such that \((g \otimes \mathrm{id})\circ \mathrm{Coconj}(y_i)\) are the images of certain elements $b_i$ of $\frac R{s_0} \otimes K$ under the canonical $R$-module homomorphism $ \tfrac R{s_0} \otimes K \to S^{-1}R \otimes K.$

Notice that the elements of $\tfrac R {s_0} \otimes K$ also can be multiplied with each other via the obvious multiplication homomorphisms
\[ \tfrac R {s} \otimes K \otimes \tfrac R {s'} \otimes K \to \tfrac R {ss'} \otimes K,\text{ where } s,s'\in S. \]
Using these homomorphisms we can compute the values $f_j(b_1,\ldots, b_n)$ of the polynomials $f_j$'s on $b_i$'s.
It is clear that these values can be interpreted as elements of $\frac R {s_0^N}\otimes K$ and that they are mapped to $0$ under the homomorphism $\frac R {s_0^N} \otimes K \to S^{-1}R \otimes K$. Consequently, there is an element $s_0' \in S$ annihilating all of them simultaneously. Set $s_1 = s_0^N s_0'$. It is clear that in $\frac R{s_1} \otimes K$ we have the identities \(f_j(b_1, \ldots, b_n) = 0\) for all \(j\)'s. Although $b_i$'s are not determined uniquely, their choice is ``almost'' unique in the sense that for any other collection of elements $\{b_i'\}$ fitting in the left square of the diagram below there is $s_2\in S$ making the the whole diagram commute:
\begin{equation} \label{eq:bis} \begin{aligned} \xymatrix{ \{y_i\} \ar@<3.5pt>[r]^(0.4){\{b_i\}} \ar@<-3.5pt>[r]_(0.4){\{b_i'\}} \ar@{^{(}->}[d] & \tfrac R {s_1} \otimes K \ar[rd] \ar[rr]^{\mathrm{can}\otimes \mathrm{id}} & & \tfrac R {s_1 s_2} \otimes K \ar[ld] \\ K \ar[rr]^(.4){(g\otimes \mathrm{id})\circ \mathrm{Coconj}} & & S^{-1} R \otimes K. }\end{aligned} \end{equation}

Now we are ready to construct a premorphism $\theta$ representing the ``conjugation by $g$'' automorphism of pro-sets
\begin{equation} \label{eq:conj-g} \up g{(-)} \colon \GG^{(\infty)}(\Phi, R) \to \GG^{(\infty)}(\Phi, R).\end{equation}
By definition, this amounts to constructing a collection of maps
\[\theta_{s}\colon \Rng(K, R^{(\theta^*(s))}) \to \Rng(K, R^{(s)})\]
for some function $\theta^* \colon S \to S$. Set $\theta^*(s) = ss_1$.

Observe that for every \(s, s' \in S\) the mapping $ \frac rs \otimes {r'}^{(ss')} \mapsto (rr')^{(s')}$ defines an isomorphism \(\frac Rs \otimes_R R^{(ss')} \cong R^{(s')}\) of $R$-modules. In particular, the \(R\)-module \(\frac Rs \otimes_R R^{(ss')}\) is equipped with a natural structure of an rng.

Let $h$ be an element of $\Rng(K, R^{(ss_1)})$.
Notice that the images of $b_i$ under $\mathrm{id}\otimes h$ satisfy all $f_j$'s.
Now the element $\theta_s(h)$ now can be constructed as the composite of $\mathrm{Coconj}$ and the unique rng homomorphism $K \to R^{(s)}$ making the following diagram commute:
\[ \xymatrix{ \{ y_i\} \ar@{^{(}->}[d] \ar[r]^(0.45){\{b_i\}} & \tfrac R {s_1} \otimes K \ar[d]^{\mathrm{id} \otimes h} \\
               K \ar@{-->}[r]_(0.33){\theta_s(h)} & \tfrac R {s_1} \otimes_R R^{(ss_1)} \ar@{=}[r] & R^{(s)}. } \]

It is clear that the morphism~\eqref{eq:conj-g} defined by $\theta$ is independent of the choices of $s_0$ and $s_1$. The diagram~\eqref{eq:bis} also shows that it is also independent of the choice of $b_i$'s.

Now let us check that \(\up g{(-)}\) is a morphism of pro-groups.
By a standard result on Hopf algebras and using the fact that $K$ is flat, we may choose the generators $y_i$ of \(K\) in such a way that $Y = \bigoplus_i \mathbb Z y_i \leq K$ is a finite free submodule and $\Delta(Y) \leq Y \otimeshat Y$. Consequently, there is a well-defined group homomorphism $Y \to \frac R{s_1} \otimes K$ given by $y_i \mapsto b_i$. From the identity $\up g{(xy)} = \up gx\, \up gy$ reformulated in the language of Hopf algebras we conclude that for sufficiently large $s_2 \in S$ the top pentagon in the following diagram commutes:
\begin{equation}\label{eq:house} \begin{aligned} \xymatrix@R=48pt@C=72pt@!0{
 & Y \ar[dr] \ar[dl]_(0.4){\Delta} & \\
Y \otimeshat Y \ar[d] &  & \frac R{s_1} \otimes K \ar[d]^{\mathrm{can} \otimes \Delta} \\
(\frac R{s_1} \otimes K)^{\otimeshat 2} \ar[rr]^{v} \ar[d]^{(\mathrm{id} \otimes h_1) \otimeshat (\mathrm{id} \otimes h_2)} && \frac R{s_1^2 s_2} \otimes K^{\otimeshat 2} \ar[d]^{\mathrm{id} \otimes (h_1 \otimeshat h_2)} \\
(R^{(ss_1s_2)})^{\otimeshat 2} \ar[d]^{m} && \frac R{s_1^2 s_2} \otimes (R^{(ss_1^2 s_2)})^{\otimeshat 2} \ar[d]^{\mathrm{id}\otimes m} \\
R^{(ss_1s_2)} \ar[r] & R^{(s)} & \frac R{s_1^2 s_2} \otimes R^{(ss_1^2 s_2).} \ar[l]
}\end{aligned}\end{equation}
The homomorphism $v$ in the above diagram is induced by canonical homomorphisms $\tfrac R {s_1} \to \tfrac R {s_1^2s_2}\text{ and }\tfrac{R} {s_1} \otimes \tfrac R{s_1} \to \tfrac R{s_1^2s_2}.$
It is clear that for any $h_1, h_2 \colon K \to R^{(ss_1^2 s_2)}$ the bottom square of the above diagram also commutes.
The commutativity of the whole diagram implies that the images of $\theta_{ss_1s_2}(h_1h_2)$ and $\theta_{ss_1s_2}(h_1)\, \theta_{ss_1s_2}(h_2)$ in $\GG_\mathrm{sc}(\Phi, R^{(s)})$ coincide. Consequently, the morphism~\eqref{eq:conj-g} is a morphism of group of objects in $\Pro(\Set)$ and hence a morphism of pro-groups (cf. the beginning of~\cref{sec:group-objects}).

It is not hard to check that $\up 1{(-)}$ is the identity automorphism of $\mathrm{id}_{\GG^{(\infty)}(\Phi, R)}$ and that for all \(g, g' \in \GG_{\mathrm{sc}}(\Phi, S^{-1} R)\) one has $\up{gg'}{(-)} = \up{g}{(\up{g'}{(-)})}$.

\subsection{The case of Steinberg groups}
Clearly, the set $\Pro(\Group)(R^{(\infty)})$ of endomorphisms of the additive pro-group of the pro-rng $R^{(\infty)}$ can be endowed with the structure of an associative ring: the addition is 
pointwise, while the multiplication is the usual composition of endomorphisms.

For every class of fractions \([\frac rs] \in S^{-1} R\) there is a well-defined pro-group endomorphism \(m_{[\frac rs]} \colon R^{(\infty)} \to R^{(\infty)}\) given by \(m_{[\frac rs]}^*(s') = ss'\) and \(m_{[\frac rs]}^{(s')}(a^{(ss')}) = (ra)^{(s')}\). It is easy to see that $m_{[\frac{r}{s}]}$ does not depend on the choice of the representative \(\tfrac r s\). Moreover, it is easy to see that \(m\) is actually a homomorphism of rings \(S^{-1} R \to \Pro(\Group)(R^{(\infty)})\) and that 
\begin{equation} \label{eq:m-mult} m_{[\frac rs]}(a^{(\infty)} b^{(\infty)}) = m_{[\frac rs]}(a^{(\infty)}) b^{(\infty)} = a^{(\infty)} m_{[\frac rs]}(b^{(\infty)}).\end{equation}
For shortness we write \(\frac rs a^{(\infty)}\) instead of \(m_{[\frac rs]}(a^{(\infty)})\). 

Our next step is to construct the action of the Steinberg group $\St(\Phi, S^{-1}R)$ on the Steinberg pro-group $\St^{(\infty)}(\Phi, R)$.
\begin{df}\label{root-action}
 Let \(R\) be a ring, \(S \subseteq R\) be a multiplicative subset, and \(\Phi\) be a root system of rank \(\geq 3\) different from \(\rB_\ell\) and \(\rC_\ell\).
 
 For $u \in S^{-1}R$ and a root subgroup morphism $x_\beta$ of $\St^{(\infty)}(\Phi\setminus\{-\alpha\}, R)$ we define the morphism 
 $\up{x_\alpha(u)}x_\beta \colon R^{(\infty)} \to \St^{(\infty)}(\Phi\setminus\{-\alpha\}, R)$
 via one of the following identities:
 \begin{align} 
 \multispan2{$\up{x_\alpha(u)}{x_\beta(b^{(\infty)})} := x_\beta(b^{(\infty)})$ \hfill if $\alpha + \beta \notin \Phi \cup \{0\}$;} \label{eq:action-R2} \\
 \up{x_\alpha(u)}{x_\beta(b^{(\infty)})}
 &:= x_{\alpha + \beta}\bigl(N_{\alpha, \beta} u b^{(\infty)}\bigr)\, x_\beta(b^{(\infty)}) \label{eq:action-R3} \\
 \multispan2{\hfill if $\alpha + \beta \in \Phi$ but $\alpha + 2\beta, 2\alpha + \beta \notin \Phi$;} \nonumber \\
 \up{x_\alpha(u)}{x_\beta(b^{(\infty)})}
 &:= x_{\alpha + \beta}\bigl(N_{\alpha, \beta} u b^{(\infty)}\bigr)\,
  x_{2\alpha + \beta}\bigl(N^{2,1}_{\alpha, \beta} u^2 b^{(\infty)}\bigr)\,
  x_\beta(b^{(\infty)}) \label{eq:action-R4a} \\
 \multispan2{\hfill if $\alpha + \beta, 2\alpha + \beta \in \Phi$;} \nonumber \\
 \up{x_\alpha(u)}{x_\beta(b^{(\infty)})}
 &:= x_{\alpha + \beta}\bigl(N_{\alpha, \beta} u b^{(\infty)}\bigr)\,
  x_{\alpha + 2\beta}\bigl(N^{1,2}_{\alpha, \beta} u {b^{(\infty)}}^2\bigr)\,
  x_\beta(b^{(\infty)}) \label{eq:action-R4b} \\
 \multispan2{\hfill if $\alpha + \beta, \alpha + 2\beta \in \Phi$.} \nonumber
 \end{align} 
\end{df}

\begin{lemma} \label{endomor-elim}
 The morphisms $\up{x_\alpha(u)}x_\beta(b^{(\infty)})$ defined above are pro-group morphisms. They give rise to an endomorphism 
 \[\up{x_\alpha(u)}(-) \colon \St^{(\infty)}(\Phi\setminus\{-\alpha\}, R) \to \St^{(\infty)}(\Phi\setminus\{-\alpha\}, R).\]
\end{lemma}
\begin{proof}
 By~\cref{rem:xpma-presentation} we need to verify the identities
 \begin{align}
 \up{x_\alpha(u)} x_\beta(b^{(\infty)}) \cdot \up{x_\alpha(u)} x_\beta(c^{(\infty)}) & = \up{x_\alpha(u)} x_\beta(b^{(\infty)} + c^{(\infty)}), \label{eq:additivity} \\
 [\up{x_\alpha(u)}{x_\beta(b^{(\infty)})}, \up{x_\alpha(u)}{x_\gamma(c^{(\infty)})}] & = \prod_{\substack{i\beta + j\gamma \in \Phi\\ i, j > 0}} \up{x_\alpha(u)}{x_{i\beta + j\gamma}\bigl(N_{\beta, \gamma}^{i, j} {b^{(\infty)}}^i {c^{(\infty)}}^j \bigr)}, \label{eq:main-equation} \end{align}
 in which \(\beta, \gamma \in \Phi \setminus \{- \alpha\}\), $\beta \neq -\gamma$, $-\alpha\not\in(\ZZ_{>0}\beta + \ZZ_{>0}\gamma) \cap \Phi$ and, moreover,
  $\beta+\gamma\neq -2\alpha$. A direct consideration of root systems $\rA_2$ and $\rB_2$ shows that under these assumptions the root subset $\Sigma = \Phi\cap (\ZZ_{\geq 0}\beta + \ZZ_{\geq 0}\gamma + \ZZ_{\geq 0}\alpha)$ is special (in particular, it does not contain $-\alpha$).
  
 Set $T = \ZZ[b,c,u]$, $I := \langle b, c \rangle \trianglelefteq T$. Denote by $\mathcal{X}_{\Sigma, I}$ the set $\{ x_\delta(d) \mid d\in I,\ \delta\in\Sigma \}$ of defining generators of $\mathrm{U}(\Sigma, I)$ and by $\mathcal{R}_{\Sigma, I}$ the normal subgroup of the free group $F:=F(\mathcal{X}_{\Sigma, I})$ generated by the relations $R_\beta(b, c) = x_\beta(b)x_\beta(c)x_\beta(b + c)^{-1}$ of type~\eqref{R1} and relations $R_{\beta,\gamma}(b,c)$ of type \eqref{R2}--\eqref{R4} with $\beta,\gamma \in \Sigma$. By~\cref{rem:uni-rad} the group $\mathrm{U}(\Sigma, I)$ is isomorphic to $F(\mathcal{X}_{\Sigma, I})/\mathcal{R}_{\Sigma, I}$.
 
 Notice that there is a well-defined conjugation action of $\mathrm{U}(\Sigma, T)$ on $\mathrm{U}(\Sigma, I)$. The action of $x_\alpha(u)$ on each $x_\beta(b) \in \mathcal{X}_{\Sigma, I}$ is given precisely by formulas~\eqref{eq:action-R2}--\eqref{eq:action-R4b} (with the superscript $(\infty)$ removed). The fact that this endomorphism of $F$ defines an endomorphism of $\mathrm{U}(\Sigma, I)$ implies that for any Steinberg relation $R$ (which is either $R_\beta(b,c)$ or $R_{\beta,\gamma}(b,c)$) the word $\up{x_\alpha(u)}R$ can be rewritten in $F$ as a product $\prod_i R_i^{g_i}$, where $R_i$ is either $R_{\beta_i}(f_{i}^1(b,c,u), f_{i}^2(b,c,u))$ or $R_{\beta_i,\gamma_i}(f^{1}_i(b,c,u), f^{2}_i(b,c,u))$, $f^1_i, f^2_i \in I$, $\beta_i, \gamma_i \in \Sigma$ and $g_i \in F$.
 
 Notice that a polynomial $f(b,c,u)\in I$ and $u_0 = [\tfrac r s] \in S^{-1}R$ give rise to a morphism $R^{(\infty)}\times R^{(\infty)} \to R^{(\infty)}$ of pro-sets.
 Indeed, this morphism is the result of reading the expression $f(b^{(\infty)}, c^{(\infty)}, u_0)$ according to~\cref{conv:notation}. It can be presented as a certain composition of morphisms $\Delta$, $m_{[\tfrac r s]}$, $m$, $+$ and morphisms of component reordering, cf.~\cref{example-commutator}. It follows from the previous paragraph that by inserting and deleting trivial subexpressions of the form $x_{\delta}(f)x_{\delta}(f)^{-1}$, $f = f(b^{(\infty)}, c^{(\infty)}, u_0)$ we can rewrite either of the equations \eqref{eq:additivity}--\eqref{eq:main-equation} into a product of conjugates of words \(R_{\beta_i}(f_{i}^1(b^{(\infty)},c^{(\infty)},u_0),\- f_{i}^2(b^{(\infty)},c^{(\infty)},u_0))\) or $R_{\beta_i,\gamma_i}(f^{1}_i(b^{(\infty)}, c^{(\infty)},u_0),\- f^{2}_i(b^{(\infty)},c^{(\infty)},u_0))$.
 The assertion of the lemma now follows from the fact that the root subgroup morphisms of $\St^{(\infty)}(\Phi\setminus\{-\alpha\}, R)$ obviously satisfy all the relations $R_\beta$, $R_{\beta, \gamma}$ for $\beta,\gamma \in \Sigma$.
\end{proof}

\begin{prop}\label{SteinbergLocalAction}
 Let $R$ and $\Phi$ be as in the above definition.
 The identities~\eqref{eq:action-R2}--\eqref{eq:action-R4b} specify a unique well-defined conjugation action of the group \(\St(\Phi, S^{-1} R)\) on the pro-group \(\St^{(\infty)}(\Phi, R)\).
 The morphism \(\mathrm{st} \colon \St^{(\infty)}(\Phi, R) \to \GG^{(\infty)}(\Phi, R)\) is equivariant with respect to this action.
\end{prop}
\begin{proof}
 By~\cref{SingleRootElimination} the pro-group $\St^{(\infty)}(\Phi\setminus\{-\alpha\}, R)$ is isomorphic to $\St^{(\infty)}(\Phi, R)$, therefore from \cref{endomor-elim} we obtain an endomorphism of $\St^{(\infty)}(\Phi,\- R)$, for which we use the same notation $\up{x_\alpha(u)}(-)$. It is easy to see that the mapping $u \mapsto \up{x_\alpha(u)}(-)$, in fact, specifies a group homomorphism $S^{-1}R \to \Aut(\St^{(\infty)}(\Phi, R))$. 
 
 The next step of the proof is to verify that the constructed action satisfies the Steinberg relations~\eqref{R2}--\eqref{R4} defining the group $\St(\Phi, S^{-1}R)$. Let $R_{\alpha, \beta}(a, b)$ be one of these relations (as in the proof of the above lemma we interpret it as an element of the free group on generators $x_\alpha(a)$). Denote by $\Phi_0$ the root subsystem of $\Phi$ generated by $\alpha$ and $\beta$ and let $\gamma$ be a root of $\Phi \setminus \Phi_0$. Since the root subset $\Sigma = (\ZZ_{\geq 0} \alpha + \ZZ_{\geq 0}\beta + \ZZ_{\geq 0}\gamma)\cap \Phi$ is special, we can use an argument similar to the proof of~\cref{endomor-elim} in order to shows that the result of conjugation of $x_\gamma$ with $R_{\alpha, \beta}(a, b)$ coincides with $x_\gamma$. By~\cref{DoubleRootElimination} this implies that the automorphism of conjugation with $R_{\alpha,\beta}(a, b)$ coincides with the trivial automorphism of $\St^{(\infty)}(\Phi, R)$.

 Now let us verify the equivariance. It is enough to show that
 \begin{equation}\label{eq:equivariance-statement} \mathrm{st}\bigl(\up{x_\alpha(u)}{x_\beta(b^{(\infty)})}\bigr) = \up{t_\alpha(u)}{\mathrm{st}\bigl(x_\beta(b^{(\infty)})\bigr)},\text{ for all }\alpha,\beta \in \Phi, \beta \neq -\alpha.\end{equation}
 We will verify the equality in the case when \(\alpha + \beta \in \Phi\) but \(\alpha + 2\beta, 2\alpha + \beta \notin \Phi\), the proof in the other cases being similar.
 Consider the elementary root unipotent \(t_\beta(-) \colon \mathbb{G}_a \to \GG_{\mathrm{sc}}(\Phi, -)\) as a morphism of algebraic groups.
 Denote by \(\mathrm{Cot}_\beta \colon K \to X\ZZ[X]\) the corresponding homomorphism of augmentation ideals of Hopf algebras (as in~\cref{sec:local-Chevalley} we abbreviate $\Ker(\varepsilon)$ to $K$).
 Fix \(u \in S^{-1} R\). Since root unipotents $t_\alpha$ satisfy~\eqref{R3} the outer square of the following diagram commutes:
 \[\scalebox{0.8}{\begin{minipage}{1.1\textwidth}\[\xymatrix@R=2pc {
 K \ar[rr]^(.42){(t_\alpha(u)\otimes\mathrm{id})\,\mathrm{Coconj}} \ar[ddd]^{\Delta}  & & 
 S^{-1}R \otimes K \ar[rr]^(0.48){\mathrm{id}\otimes \mathrm{Cot}_\beta} & &
 S^{-1}R \otimes X \mathbb Z[X]\\
 & Y \ar[d] \ar@{^{(}->}[ul] \ar[r]^(0.4){\{b_i\}} & \tfrac R {s_1} \otimes K \ar[u] \ar[r]^(0.42){\mathrm{can}\otimes\mathrm{Cot}_\beta} & \tfrac R {s_1^2s_2} \otimes X\ZZ[X] \ar[ru]\\ 
 & Y \otimeshat Y  \ar@{^{(}->}[dl] \ar@{-->}[rr] & & (\tfrac R {s_1} \otimes X\ZZ[X])^{\otimeshat 2} \ar[r] \ar[u]^{v} & (S^{-1}R \otimes X\ZZ[X])^{\otimeshat 2} \ar[uu]^{m} \\
 K \otimeshat K \ar[rrrr]^{\mathrm{Cot}_\beta \otimeshat \mathrm{Cot}_{\alpha + \beta}} & & & &
 X_1 \mathbb Z[X_1] \otimeshat X_2 \mathbb Z[X_2]. \ar[u]^{\begin{psmallmatrix*}[l] X_1 & \mapsto & 1 \otimes X\\ X_2 & \mapsto & N_{\alpha,\beta} u \otimes X \end{psmallmatrix*}}  &
 }\]\end{minipage}}\]
 Similarly, to the construction of~\eqref{eq:house} we can choose a free abelian subgroup $Y\leq K$ of finite rank generating $K$ multiplicatively so that $\Delta(Y)\leq Y\otimeshat Y$. 
 Further, following the procedure described in~\eqref{eq:bis} we can choose $s_1$ and $\{b_i\}$ in such a manner that the top left square of the above diagram commutes. Also, there is a dashed arrow making the bottom pentagon commute. Finally, there exists sufficiently large $s_2$ such that the central square of the diagram commutes. This commutativity implies~\eqref{eq:equivariance-statement} (the top path gives the right-hand side while the bottom one gives the left-hand side; the homomorphism $v$ is has the same meaning as in~\eqref{eq:house}). Thus, by \cref{SteinbergPresentation} the morphism \(\mathrm{st} \colon \St^{(\infty)}(\Phi, S^{-1} R) \to \GG^{(\infty)}(\Phi, R)\) is \(\St(\Phi, S^{-1} R)\)-equivariant.
\end{proof}

\begin{theorem}\label{ChevalleyLocalAction}
 Let \(R\) be a ring, \(M \trianglelefteq R\) be a maximal ideal and \(\Phi\) be a root system of rank \(\geq 3\) different from \(\rB_\ell\) and \(\rC_\ell\). Set \(S = R \setminus M\). Then the action of $\St(\Phi, R_M)$ defined in~\cref{SteinbergLocalAction} gives rise to an action of the group \(\GG_{\mathrm{sc}}(\Phi, R_M)\) on \(\St^{(\infty)}(\Phi, R)\). The morphism \(\mathrm{st} \colon \St^{(\infty)}(\Phi, R) \to \GG^{(\infty)}(\Phi, R)\) is \(\GG_{\mathrm{sc}}(\Phi, R_M)\)-equivariant.
\end{theorem}
\begin{proof}
 Recall that for \(a, b \in R_M^\times\) and \(\alpha \in \Phi\) one can define the following elements of $\St(\Phi, R_M)$: 
  \begin{align*} w_\alpha(a) & =  x_\alpha(a) \cdot x_{-\alpha}(-a^{-1}) \cdot x_\alpha(a), \\
                 h_\alpha(a) & =  w_\alpha(a) \cdot w_\alpha(1)^{-1}, \\
                 \{ a,\-b \}_\alpha & = h_\alpha(ab) \cdot h_\alpha^{-1}(a) \cdot h_\alpha^{-1}(b). \end{align*}
 Since $R_M$ is local, the group $\GG_{\mathrm{sc}}(\Phi, R_M)$ coincides with its elementary subgroup $\E_{\mathrm{sc}}(\Phi, R)$.
 Consquently, by~\cite[Proposition~1.6]{Abe69} and \cite[Theorem~2.13]{Ste73} the group \(\GG_{\mathrm{sc}}(\Phi, R_M)\) is the quotient of \(\St(\Phi, R_M)\) by the relations $\{a,\-b\}_\alpha = 1$ for \(a, b \in R_M^\times\) and some fixed long root $\alpha\in\Phi$. Thus, it remains to check that the Steinberg symbols $\{a,\-b\}_\alpha$ act trivially on \(\St^{(\infty)}(\Phi, R)\), which, in turn, follows from the relations
 \begin{equation}\label{eq:h-conj} \up{h_\alpha(a)}x_\beta(b^{(\infty)}) = x_\beta(a^{\langle \beta, \alpha \rangle}b^{(\infty)}),\ \beta\neq\pm\alpha. \end{equation}
 
To prove~\eqref{eq:h-conj} take a root \(\beta \in \Phi\) linearly independent with \(\alpha\).
 Direct computation using~\cref{root-action} shows that
 \[\up{h_\alpha(a)}{x_\beta(b^{(\infty)})} = \prod_{\substack{i\alpha + j\beta \in \Phi\\ j > 0}} x_{i\alpha + j\beta}(P_{i, j}(a) \cdot {b^{(\infty)}}^j)\]
 for some Laurent polynomials \(P_{i, j}(t) \in \mathbb Z[t, t^{-1}]\) depending only on \(\alpha\) and \(\beta\) and the chosen order of factors. Since in the Steinberg group \(\St(\Phi, \mathbb Z[t, t^{-1}])\) we have a similar identity
 \[\up{h_\alpha(t)}{x_\beta(1)} = \prod_{\substack{i\alpha + j\beta \in \Phi\\ j > 0}} x_{i\alpha + j\beta}(P_{i, j}(t)) = x_{\beta}(t^{\langle \beta,\-\alpha \rangle}),\]
 we conclude by~\cref{rem:uni-rad} that \(P_{i, j}(t) = 0\) with the sole exception of \(P_{0, 1}(t) = t^{(\beta, \alpha)}\). Thus, the proof of~\eqref{eq:h-conj} is complete.
\end{proof}

\section{Proof of the main result} \label{sec:proof-main}
\subsection{The construction of a crossed module}
In this section we prove the main results of the paper, namely we construct a crossed module structure on the homomorphism $\mathrm{st}$ and prove the centrality of $\mathrm{K}_2$. Let us, first, briefly recall the definition of a crossed module. 

\begin{df} \label{df:crossed-module}
Let $N$ be a group acting on itself by left conjugation.
A group homomorphism $\varphi\colon M \to N$ is called a {\it crossed module} if one can define the action of the group $N$ on $M$ in such a way that that $\varphi$ preserves the action of $N$ and, moreover, the identity ${}^{\varphi(m)}\!m' = m m' m^{-1}$, called {\it Peiffer identity}, holds for all $m, m' \in M$.
\end{df}
It is an easy exercise to check that the kernel of $\varphi$ is always a central subgroup of $M$, while the image of $\varphi$ is normal in $N$.

For the rest of this section $R$ is an arbitrary commutative ring and $M$ is a maximal ideal of $R$.
As before, \(\Phi\) is a root system of rank \(\geq 3\) different from \(\rB_\ell\) and \(\rC_\ell\).

 We denote by $S$ the multiplicative system associated to $M$, i.\,e. $S= R\setminus M$. By~\cref{ChevalleyLocalAction} the group $\GG_{\mathrm{sc}}(\Phi, R_M)$ acts on both $\St^{(\infty)}(\Phi, R)$ and $\GG^{(\infty)}(\Phi, R)$ by conjugation and this action is preserved by the morphism $\mathrm{st}$. Consequently, $\GG_{\mathrm{sc}}(\Phi, R)$ act one these pro-groups by their images in $\GG_\mathrm{sc}(\Phi, R_M) $. Consider the following canonical morphisms 
\begin{align*}
 \pi_R \colon R^{(\infty)} &\to R,\\
 \pi_{\St} \colon \St^{(\infty)}(\Phi, R) &\to \St(\Phi, R),\\
 \pi_{\GG} \colon \GG^{(\infty)}(\Phi, R) &\to \GG_{\mathrm{sc}}(\Phi, R).
\end{align*}

 Notice that the action of $h\in \GG_{\mathrm{sc}}(\Phi, R)$ on the pro-group $\GG^{(\infty)}(\Phi, R)$ is compatible with the conjugation action of $\GG_{\mathrm{sc}}(\Phi, R)$ on itself, i.\,e. $\up{h}\pi_{\GG}(g^{(\infty)}) = \pi_{\GG}(\up{h} {g^{(\infty)}})$.
Similarly, the conjugation action of $h\in \St(\Phi, R)$ on $\St^{(\infty)}(\Phi, R)$ is compatible with the conjugation action of $\St(\Phi, R)$ on itself, i.\,e. $\up{h}\pi_{\St}(g^{(\infty)}) = \pi_{\St}(\up{h}{g^{(\infty)}})$.
Notice also that $x_\alpha \pi_R = \pi_{\St} x_\alpha$ and $\pi_{\GG}\, \mathrm{st} = \mathrm{st}\, \pi_{\St}$.

\begin{lemma}\label{CentralityK2}
 The subgroup \(\mathrm K_2(\Phi, R) = \Ker(\mathrm{st} \colon \St(\Phi, R) \to \GG(\Phi, R))\) is contained in the center of \(\St(\Phi, R)\).
\end{lemma}
\begin{proof}
 Let $g$ be an element of $\mathrm K_2(\Phi, R)$ and $\alpha \in \Phi$ be a root. Consider the ideal
 \[I_\alpha = \{a \in R \mid \up g {x_\alpha(ra)} = x_\alpha(ra) \text{ for all } r \in R\}.\]
 We need to show that \(I_\alpha = R\). Let $M \trianglelefteq R$ be a maximal ideal.
 By Theorem~2, the element \(g\) acts trivially on the pro-group \(\St^{(\infty)}(\Phi, R)\) associated with $S = R \setminus M$. Consequently, we obtain the following equality of morphisms $R^{(\infty)} \to \St(\Phi, R)$:
 \begin{multline*}
  \up g {x_\alpha(\pi_R(a^{(\infty)}))}
  = \up g{\pi_{\St}(x_\alpha(a^{(\infty)}))}
  = \pi_{\St}(\up g{x_\alpha(a^{(\infty)})}) = \\
  = \pi_{\St}(x_\alpha(a^{(\infty)}))
  = x_\alpha(\pi_R(a^{(\infty)})),
 \end{multline*}
 which implies that \(I_\alpha \not \subseteq M\).
\end{proof}

With similar technique, we also can re-prove Taddei's normality theorem (see~\cite{Ta86}).
\begin{lemma}\label{Normality}
 The subgroup \(\mathrm E(\Phi, R) = \mathrm{Im}(\mathrm{st} \colon \St(\Phi, R) \to \GG_{\mathrm{sc}}(\Phi, R))\) is normal in \(\GG_{\mathrm{sc}}(\Phi, R)\).
\end{lemma}
\begin{proof}
 Let $g$ be an element of $\GG_{\mathrm{sc}}(\Phi, R)\) and \(\alpha \in \Phi\) be a root. Consider the ideal
 \[I_\alpha = \{a \in R \mid \up g{t_\alpha(ra)} \in \mathrm E(\Phi, R)\text{ for all }r\in R\}.\]
 Again, we need to show that \(I_\alpha = R\). Let \(M \trianglelefteq R\) be a maximal ideal.
 Since \(g\) acts on the pro-groups \(\St^{(\infty)}(\Phi, R)\) and $\GG^{(\infty)}(\Phi, R)$ associated with the subset $S = R \setminus M$ and 
 this action is preserved by $\mathrm{st}$, we obtain the following equalities of morphisms $R^{(\infty)} \to \GG_{\mathrm{sc}}(\Phi, R)$:
 \[
  \up g{t_\alpha(\pi_R(a^{(\infty)}))}
  = \up g{\pi_{\GG}(\mathrm{st}(x_\alpha(a^{(\infty)})))}
  = \pi_{\GG}(\mathrm{st}(\up g{x_\alpha(a^{(\infty)})}))
  = \mathrm{st}(\pi_{\St}(\up g{x_\alpha(a^{(\infty)})})),
 \]
 which implies that \(I_\alpha \not \subseteq M\).
\end{proof}

\begin{theorem}\label{SteinbergCrossedModule}
 The homomorphism \(\mathrm{st} \colon \St(\Phi, R) \to \GG_{\mathrm{sc}}(\Phi, R)\) can be turned into a crossed module in a unique way.
\end{theorem}
\begin{proof}  
 Notice that by~\cref{CentralityK2} we already know that $\mathrm K_2(\Phi, R)$ is a central subgroup of $\St(\Phi, R)$. Under our assumptions the group $\St(\Phi, R)$ is a perfect central extension of $\E_{\mathrm{sc}}(\Phi, R)$, therefore by~\cite[Lemma~1.1]{St71} for every $g\in \GG_{\mathrm{sc}}(\Phi, R)$ there may exist at most one endomorphism $\up g (-)$ of $\St(\Phi, R)$ satisfying 
 \begin{equation} \label{eq:scm-2} g\, \mathrm{st}(h) g^{-1} = \mathrm{st} (\up g h)\text{ for all }h \in \St(\Phi, R).  \end{equation}
 Thus, to construct an action of $\GG_{\mathrm{sc}}(\Phi, R)$ on $\St(\Phi, R)$ it suffices to construct an endomorphism 
  \(\up g{(-)} \colon \St(\Phi, R) \to \St(\Phi, R)\) satisfying~\eqref{eq:scm-2}.
 Also, this action would automatically satisfy Peiffer identity.

 Let $\alpha$ be a root of $\Phi$. Since \(\mathrm E(\Phi, R) \trianglelefteq \GG_{\mathrm{sc}}(\Phi, R)\) by~\cref{Normality}, the set
 $Y_\alpha(a) = \mathrm{st}^{-1}\bigl(g{t_\alpha(a)} g^{-1}\bigr)$ is nonempty for all \(a \in R\). 
 Moreover, $Y_\alpha(a)$ is a coset of the subgroup \(\mathrm K_2(\Phi, R)\) and \(Y_\alpha(a + a') = Y_\alpha(a) Y_\alpha(a')\).

 First of all, let us show that \([Y_\alpha(a), Y_\beta(b)] = 1\) provided \(\alpha + \beta \notin \Phi \cup \{0\}\). 
 Let \(\alpha, \beta\) be roots as above, $a$ be an element of $R$ and $h$ be an element of $Y_\alpha(a)\).
 Set $I = \{b \in R \mid [h, Y_\beta(bR)] = 1\}$ and let \(M \trianglelefteq R\) be a maximal ideal. Since
 \[[h, \pi_{\St}(\up g {x_\beta(b^{(\infty)})})] = \pi_{\St}(\up{g t_\alpha(a)} {x_\beta(b^{(\infty)})}) \pi_{\St}(\up g {x_\beta(-b^{(\infty)})}) = 1,\]
 we obtain  \(I \not \subseteq M\) and, consequently, \(I = R\), which proves the assertion.

 Now let $\alpha$ be a root of $\Phi$. By our assumptions on $\Phi$ there exist roots \(\beta, \gamma \in \Phi\) of the same length such that \(\alpha = \beta + \gamma\). Denote by \(y_\alpha(a)\) the only element of the set \([Y_\beta(N_{\beta, \gamma}), Y_\gamma(a)]\). 
 We define the map $\up g (-)$ on the set of generators of $\St(\Phi, R)$ by $\up g {x_\alpha(a)} = y_\alpha(a)$, and we claim that such a definition induces an endomorphism, i.e., that $y_{\alpha}(a)$ satisfy the Steiberg relations. 
 Clearly, \(y_\alpha(a)\) satisfy~\eqref{R1}. Since $y_\alpha(a) \in Y_\alpha(a)$, it is also clear that they satisfy relations~\eqref{R2}.
 
 It remains to check that \(y_\alpha(a)\) satisfy the Steinberg relations \eqref{R3} and \eqref{R4}.
 Let \(M \trianglelefteq R\) be a maximal ideal. Recall that \(\pi_{\St}(\up g{x_\gamma(a^{(\infty)})})\) denotes a certain pro-group morphism \(R^{(\infty)} \to \St(\Phi, R)\), so \([h, \pi_{\St}(\up g{x_\gamma(a^{(\infty)})})]\) is a pro-set morphism \(R^{(\infty)} \to \St(\Phi, R)\) for a fixed \(h \in Y_\beta(N_{\beta, \gamma})\). Since
 \[\mathrm{st}(\pi_{\St}(\up g{x_\gamma(a^{(\infty)})})) = g t_\gamma(\pi_R(a^{(\infty)})) g^{-1},\]
 we have
 \begin{align} 
  y_\alpha(\pi_R(a^{(\infty)})) &= [h, \pi_{\St}(\up g{x_\gamma(a^{(\infty)})})] \nonumber \\
  &= \pi_{\St}(\up{\mathrm{st}(h)\, g}{x_\gamma(a^{(\infty)})})\, \pi_{\St}(\up g{x_\gamma(a^{(\infty)})})^{-1} \label{eq:scm1} \\
  &= \pi_{\St}\bigl(\up{g t_\beta(N_{\beta, \gamma})}{x_\gamma(a^{(\infty)})}\, \up g{x_\gamma(-a^{(\infty)})}\bigr) = \pi_{\St}(\up g{x_\alpha(a^{(\infty)})}).\nonumber
 \end{align}
 
 Now let $a$ be an element of $R$ and \(\alpha, \beta\) be roots of $\Phi$ such that \(\alpha + \beta, 2\alpha + \beta \in \Phi\). 
 Consider the ideal
 \begin{align*}
 J = \{b \in R &\mid [y_\alpha(a), y_\beta(br)] = y_{\alpha + \beta}(N_{\alpha, \beta} abr)
 y_{2\alpha + \beta}(N_{\alpha, \beta}^{2,1} a^2 br) \text{ for all } r \in R\}.
 \end{align*}
 From~\eqref{eq:scm1} we obtain 
 \begin{align*}
 [y_\alpha(a), y_\beta(\pi_R(b^{(\infty)}))] &= [y_\alpha(a), \pi_{\St}(\up g{x_\beta(b^{(\infty)})})]\\ 
 &= \pi_{\St}\bigl(\up{g t_\alpha(a)}{x_\beta(b^{(\infty)})}\, \up g{x_\beta(-b^{(\infty)})}\bigr)\\
 &= \pi_{\St}\bigl(\up g{x_{\alpha + \beta}(N_{\alpha, \beta} a b^{(\infty)})}\, \up g{x_{2\alpha + \beta}(N_{\alpha, \beta}^{2,1} a^2 b^{(\infty)})}\bigr)\\
 &= y_{\alpha + \beta}(N_{\alpha, \beta} a \pi_R(b^{(\infty)})) y_{2\alpha + \beta}(N_{\alpha, \beta}^{2,1} a^2 \pi_R(b^{(\infty)})).
 \end{align*}
 It follows that \(J \not \subseteq M\), i.e. \(J = R\). The verification of~\eqref{R3} is similar but easier.   
 It is clear that the constructed endomorphism satisfies~\eqref{eq:scm-2}. \end{proof}
 
 \subsection{Concluding remarks} \label{no-homotopes}
 It is possible to prove the centrality of $\mathrm{K}_2$ without using the concept of a ring homotope introduced in~\cref{ring-homotope}.
 Indeed, the idea is to reduce to the case of domains using the following result inspired by \cite[Lemma~5.3]{St99}.
 \begin{lemma} Let $R$ be a commutative ring, $I$ be an ideal of $R$ and $\Phi$ be a root system of rank $\geq 3$.
 The centrality of $\mathrm{K}_2(\Phi, R)$ implies the centrality of $\mathrm K_2(\Phi, R/I)$. \end{lemma}
 \begin{proof} Consider the following commutative diagram
 \[\xymatrix{\St(\Phi, R) \ar[r]^{\rho_\mathrm{St}} \ar[d]^{\mathrm{st}} & \St(\Phi, R/I) \ar[d] \\ \mathrm{E}_{\mathrm{sc}}(\Phi,R) \ar[r]^{\rho_\mathrm{E}} & \mathrm{E}_{\mathrm{sc}}(\Phi, R/I).}\]
 We denote by $\St(\Phi, R, I)$ the normal closure in $\St(\Phi, R)$ of the subgroup generated by elements $x_\alpha(s)$, where $s\in I$, $\alpha\in\Phi$. We also denote by $\E_{\mathrm{sc}}(\Phi, R, I)$ the image of this subgroup under $\mathrm{st}$.
 It is clear that $\E_\mathrm{sc}(\Phi, R, I)$ lies in the kernel of $\rho_\mathrm{E}$.
 
 Denote by $H$ the preimage of $\mathrm{K}_2(\Phi, R/I)$ under $\rho_\mathrm{St}$. Also set \[H'=[H, \St(\Phi, R)].\] The identity~\eqref{eq:comm-mult-rhs} shows that $H'$ is a normal subgroup of $\St(\Phi, R)$. Under our assumptions the group $\St(\Phi, R)$ is perfect, therefore using Hall--Witt identity~\eqref{eq:HW} we obtain 
 \begin{multline}\nonumber H' = [H,\, \St(\Phi, R)] = \left[H,\, [\St(\Phi, R), \St(\Phi, R)]\right] \subseteq \\ \subseteq \left[[H,\, \St(\Phi, R)],\, \St(\Phi, R)\right] = [H', \St(\Phi, R)]. \end{multline} 
 Since $\mathrm{st}(H)$ is contained in the congruence subgroup $\GG_\mathrm{sc}(\Phi, R, I)$ we obtain from the standard commutator formula~\cite[Theorem~1]{Va86} the fact that $\mathrm{st}(H')\subseteq\E_\mathrm{sc}(\Phi, R, I)$. Thus, $H'\subseteq \St(\Phi, R, I) \cdot \mathrm{K}_2(\Phi, R)$. Finally, from our assumptions we obtain 
 \[H' \subseteq [H', \St(\Phi, R)] \subseteq [\St(\Phi, R, I)\cdot \mathrm{K}_2(\Phi, R),\, \St(\Phi, R)] \subseteq \St(\Phi,R,I),\] from which the assertion of the lemma follows.
 \end{proof}
 Since every commutative ring $R$ can be presented as a quotient of the polynomial ring over $\ZZ$ with possibly infinite number of variables,
  we may restrict ourselves to considering only integral domains throughout the paper.
 As noted before, in this case the homotopes $R^{(s)}$ turn into the usual principal ideals $sR \trianglelefteq R$.
 
\printbibliography
\end{document}